\documentclass{article}

\PassOptionsToPackage{numbers, compress}{natbib}

\usepackage[preprint]{neurips_2019}

\usepackage{newtxtext}
\usepackage[smallerops]{newtxmath}
\usepackage{amsmath,amsthm,amssymb}
\usepackage[algo2e,ruled]{algorithm2e}
\usepackage{graphicx}
\usepackage{color}

\usepackage[utf8]{inputenc} 
\usepackage[T1]{fontenc}    
\usepackage{hyperref}       
\usepackage{url}            
\usepackage{booktabs}       
\usepackage{amsfonts}       
\usepackage{nicefrac}       
\usepackage{microtype}      

\newtheorem{theorem}{Theorem}

\newtheorem{lemma}{Lemma}

\newtheorem{assumption}{Assumption}

\newcommand{\be}{\begin{equation}}
\newcommand{\ee}{\end{equation}}
\newcommand{\bee}{\begin{equation*}}
\newcommand{\eee}{\end{equation*}}
\newcommand{\bea}{\begin{eqnarray}}
\newcommand{\eea}{\end{eqnarray}}
\newcommand{\beaa}{\begin{eqnarray*}}
\newcommand{\eeaa}{\end{eqnarray*}}

\newcommand{\minimize}{\mathop{\textrm{minimize}}}
\newcommand{\maximize}{\mathop{\textrm{maximize}}}

\newcommand{\tcgt}{\tilde{\mathcal{G}}(x_i^t)}

\newcommand{\cgt}{\mathcal{G}(x_i^t)}

\newcommand{\E}{\mathbf{E}}
\newcommand{\R}{\mathbf{R}} 
\newcommand{\Var}{\mathbf{Var}}

\newcommand{\cB}{\mathcal{B}}

\newcommand{\cG}{\mathcal{G}}

\newcommand{\cS}{\mathcal{S}} 
\newcommand{\tnf}{\tilde{\nabla}F}    
\newcommand{\cO}{{\mathcal{O}}}

\newcommand{\prox}{\mathbf{prox}}

\newcommand{\argmin}{\mathop{\rm argmin}} 
\newcommand{\half}{\frac{1}{2}}

\title{A Stochastic Composite Gradient Method with Incremental Variance Reduction}

\author{%
  Junyu Zhang \\
  University of Minnesota \\
  Minneapolis, Minnesota 55455 \\
  \texttt{zhan4393@umn.edu} \\
  \And
  Lin Xiao \\
  Microsoft Research \\
  Redmond, Washington 98052 \\
  \texttt{lin.xiao@microsoft.com} \\
}

\begin{document}

\maketitle

\begin{abstract}
    We consider the problem of minimizing the composition of a smooth 
    (nonconvex) function and a smooth vector mapping, where the inner mapping
    is in the form of an expectation over some random variable or a finite sum. 
    We propose a stochastic composite gradient method that employs an 
    incremental variance-reduced estimator for both the inner vector mapping 
    and its Jacobian. 
    We show that this method achieves the same orders of complexity
    as the best known first-order methods for minimizing expected-value and 
    finite-sum nonconvex functions, despite the additional outer composition 
    which renders the composite gradient estimator biased.
    This finding enables a much broader range of applications in machine 
    learning to benefit from the low complexity of 
    incremental variance-reduction methods.
\end{abstract}

\section{Introduction}
\label{sec:intro}

We consider stochastic composite optimization problems of the form
\be
\label{prob:main}
\minimize_{x\in\R^d} \quad  f\bigl(\E_\xi[g_\xi(x)] \bigr) + r(x) \, ,
\ee
where $f:\R^p\to\R$ is a smooth and possibly nonconvex function, 
$\xi$ is a random variable, 
each $g_{\xi}:\R^d\to\R^p$ is a smooth vector mapping,
and~$r$ is convex and lower-semicontinuous.
A special case we will consider separately is when~$\xi$ is a discrete random 
variable with uniform distribution over $\{1,2,\ldots,n\}$.
In this case the problem is equivalent to a deterministic optimization problem
\be
\label{prob:main-finite}
\minimize_{x\in\R^d} \quad 
f\biggl(\frac{1}{n}\sum_{i=1}^n g_i(x)\biggr) + r(x) \, .
\ee
The formulations in~\eqref{prob:main} and~\eqref{prob:main-finite}
cover a broader range of applications than classical stochastic optimization 
and empirical risk minimization (ERM) problems where each $g_\xi$ is a scalar
function ($p=1$) and~$f$ is the scalar identity map.
A well-known example is policy evaluation in reinforcement 
learning (RL) \citep[e.g.,][]{sutton1998reinforcement}.
With linear value function approximation, it can be formulated as
\[
\minimize_{x\in\R^d} \quad \bigl\|\E[A]x-\E[b] \bigr\|^2,
\]
where $A$ and $b$ are random matrix and vector generated by 
a Markov decision process (MDP) \citep[e.g.,][]{dann2014policy}.
Here we have $f(\cdot)=\|\cdot\|^2$, $\xi=(A, b)$ and $g_\xi(x)=Ax-b$.

Another interesting application is risk-averse optimization
\citep[e.g.,][]{Rockafellar2007CoherentRisk,Ruszczynski2013risk-averse}, 
which has many applications in RL and financial mathematics.
We consider a general formulation of mean-variance trade-off:
\begin{equation}\label{prob:risk-averse}
\maximize_{x\in\R^d}~~ \left\{\, \E\bigl[h_\xi(x)\bigr]-\lambda\Var\bigl(h_\xi(x)\bigr)
~\equiv~ 
\E\bigl[h_\xi(x)\bigr] - \lambda \Bigl(\E\bigl[h_\xi^2(x)\bigr] - \E\bigl[h_\xi(x)\bigr]^2\Bigr) \right\}, 
\end{equation}
where each $h_\xi(x):\R^d\to\R$ is a reward function 
(such as total portfolio return).
The goal of problem~\eqref{prob:risk-averse} is to maximize the average
reward with a penalty on the variance which captures the potential risk.
It can be cast in the form of~\eqref{prob:main} by using the mappings
\begin{equation}\label{eqn:mean-var-mapping}
g_\xi(x):\R^d\to\R^2= \bigl[ h_\xi(x)~~ h^2_\xi(x) \bigr]^T , \qquad
f(y,z):\R^2\to\R = -y + \lambda y^2 - \lambda z \, .
\end{equation}
Here, the intermediate dimension is very low, i.e., $p=2$.
This leads to very little overhead in computation 
compared with stochastic optimization without composition.

Besides these applications, the composition structure 
in~\eqref{prob:main} and~\eqref{prob:main-finite} are of independent 
interest for research on stochastic and randomized algorithms. 
For the ease of notation, we define
\begin{equation}\label{eqn:simple-notation}
g(x) := \E_\xi[g_\xi(x)], \qquad
F(x) := f(g(x)), \qquad 
\Phi(x) := F(x) + r(x).
\end{equation}
In addition, let $f'$ and~$F'$ denote the gradients of~$f$ and~$F$ 
respectively, 
and $g'_\xi(x)\in\R^{p\times d}$ denote the Jacobian matrix of~$g_\xi$ at~$x$. 
Then we have 
\vspace{-1ex}
\[
    F'(x) =  \nabla \Bigl( f\bigl(\E_\xi[g_\xi(x)]\bigr)\Bigr) 
    = \Bigl(\E_\xi[g'_\xi(x)]\Bigr)^T f'\bigl(\E_\xi[g_\xi(x)]\bigr) \,.
\]
In practice, computing $F'(x)$ exactly can be very costly if not impossible.
A common strategy is to use stochastic approximation:
we randomly sample a subset $\cS$ of~$\xi$ from its distribution and let
\begin{equation}\label{eqn:gg-approx}
\tilde{g}(x) = \frac{1}{|\cS|}\sum_{\xi\in\cS}g_\xi(x), \qquad
\tilde{g'}(x) = \frac{1}{|\cS|}\sum_{\xi\in\cS}g'_\xi(x).
\end{equation}
However, $\left(\tilde{g'}(x)\right)^T\!f(\tilde{g}(x))$ is always a 
\emph{biased} estimate of $F'(x)$ unless one can replace $\tilde{g}(x)$ 
with the full expectation $\E_\xi[g_\xi(x)]$.
This is in great contrast to the classical stochastic optimization problem
\begin{equation}\label{prob:exp-non-composite}
    \minimize_{x\in\R^d} \quad \E_\xi\bigl[g_\xi(x)\bigr] + r(x) \, ,
\end{equation}
where $\tilde{g'}(x)$ in~\eqref{eqn:gg-approx} is always an unbiased gradient
estimator for the smooth part $g(x)=\E_\xi\bigl[g_\xi(x)\bigr]$.
Using biased gradient estimators can cause various difficulties for 
constructing and analyzing randomized algorithms, but is often inevitable in 
dealing with more complex objective functions other than the empirical risk
\citep[see, e.g.,][]{Pratik,Smoothing-1,Smoothing-2,Smoothing-3}. 
As a simplest model, the analysis of randomized algorithms 
for~\eqref{prob:main} may provide insights for 
solving more challenging problems.

In this paper, we develop an efficient stochastic composite gradient method
called CIVR (Composite Incremental Variance Reduction),
for solving problems of the forms~\eqref{prob:main} 
and~\eqref{prob:main-finite}.
We measure efficiency by the sample complexity of the individual
functions $g_\xi$ and their Jacobian $g'_\xi$, i.e., the total number of times
they need to be evaluated at some point, 
in order to find an $\epsilon$-approximate solution.
For nonconvex functions, an $\epsilon$-approximate solution is some 
random output of the algorithm $\bar{x}\in\R^d$ that satisfies 
$\E[\|\cG(\bar{x})\|^2]\leq\epsilon$,
where $\cG(\bar{x})$ is the \emph{proximal gradient mapping} 
of the objective function~$\Phi$ at~$\bar{x}$
(see details in Section~\ref{sec:algm}).
If $r\equiv 0$, then $\cG(\bar{x})=F'(\bar{x})$ 
and the criteria for $\epsilon$-approximation becomes
$\E[\|F'(\bar{x})\|^2]\leq\epsilon$.
If the objective~$\Phi$ is convex, we require 
$\E[\Phi(\bar{x})-\Phi^\star]\leq\epsilon$ 
where $\Phi^\star=\inf_{x}\Phi(x)$.
For smooth \emph{and} convex functions, these two notions are compatible, 
meaning that the dependence of the sample complexity on~$\epsilon$ in terms
of both notions are of the same order.

\begin{table}[t]
    \caption{Sample complexities of CIVR (Composite Incremental Variance Reduction)}
  \label{tab:complexities}
  \centering
  \begin{tabular}{cccc}
    \toprule
    & \multicolumn{3}{c}{Assumptions (common: $f$ and $g_\xi$ Lipschitz and smooth, thus $F$ smooth)} \\
    \cmidrule(l){2-4}
    Problem  & $F$ nonconvex & $F$ $\nu$-gradient dominant & $F$ convex, $r$ convex \\
             & $r$ convex             & $r\equiv 0$           & $\Phi$ $\mu$-optimally strongly convex \\
    \midrule
    \eqref{prob:main}            & $\cO\bigl(\epsilon^{-3/2}\bigr)$  
    & $\cO\left(\bigl(\nu\epsilon^{-1}\bigr)\log\epsilon^{-1}\right)$ 
    & $\cO\left(\bigl(\mu^{-1}\epsilon^{-1}\bigr)\log\epsilon^{-1}\right)$ \\[1ex]
    \eqref{prob:main-finite}     
    & $\cO\bigl(\min\{\epsilon^{-3/2},\,n^{1/2}\epsilon^{-1}\}\bigr)$ 
    & $\cO\left(\bigl(n+\nu n^{1/2}\bigr)\log\epsilon^{-1}\right)$
    & $\cO\left(\bigl(n+\mu^{-1} n^{1/2}\bigr)\log\epsilon^{-1}\right)$  \\
    \bottomrule
  \end{tabular}
\end{table}

Table~\ref{tab:complexities} summarizes the sample complexities of the CIVR
method under different assumptions obtained in this paper.
We can define a condition number $\kappa=\cO(\nu)$ for $\nu$-gradient dominant functions and $\kappa=\cO(1/\mu)$ for $\mu$-optimally strongly convex functions, then the complexities become 
$\cO\bigl(\bigl(\kappa\epsilon^{-1}\bigr)\log\epsilon^{-1}\bigr)$ and
$\cO\bigl(\bigl(n+\kappa n^{1/2}\bigr)\log\epsilon^{-1}\bigr)$
for~\eqref{prob:main} and~\eqref{prob:main-finite} respectively.
In order to better position our contributions,
we next discuss related work and then putting these results into context.

\subsection{Related Work} 
We first discuss the nonconvex stochastic optimization 
problem~\eqref{prob:exp-non-composite},
which is a special cases of~\eqref{prob:main}.
When $r\!\equiv\! 0$ and $g(x)\!=\!\E_\xi[g_\xi(x)]$ is smooth, 
\citet{GhadimiLan2013} developed a randomized stochastic gradient method 
with iteration complexity $\cO(\epsilon^{-2})$. 
\citet{Natasha2NIPS2018} obtained $\cO\bigl(\epsilon^{-1.625}\bigr)$
with additional second-order guarantee.
There are also many recent works on solving its finite-sum version 
\begin{equation}\label{prob:erm-finite-sum}
\minimize_{x\in\R^d} \quad \frac{1}{n}\sum_{i=1}^n g_i(x) + r(x),
\end{equation}
which is a special case of~\eqref{prob:main-finite}.
By extending the variance reduction techniques 
SVRG \citep{johnson2013accelerating,xiaozhang2014proxsvrg}
and SAGA \citep{defazio2014saga} to nonconvex optimization, 
\citet{Allen-ZhuHazan2016} and
\citet{Reddi2016SVRGnonconvex,NCVX-SAGA-Smooth,NCVX-SAGA-Nonsmooth}
developed randomized algorithms with sample complexity
$\cO(n+n^{2/3}\epsilon^{-1})$.
Under additional assumptions of gradient dominance or strong convexity,
they obtained sample complexity
$\cO((n+\kappa n^{2/3})\log\epsilon^{-1})$, 
where $\kappa$ is a suitable condition number.
\citet{Natasha2017ICML}  and \citet{LeiJuChenJordan2017} obtained 
$\cO\bigl(\min\{\epsilon^{-5/3},\,n^{2/3}\epsilon^{-1}\}\bigr)$.

Based on a new variance reduction technique called SARAH \citep{SARAH2017ICML},
\citet{SmoothSARAH2019} and \citet{ProxSARAH2019} 
developed nonconvex extensions to obtain sample complexities  
$\cO\bigl(\epsilon^{-3/2}\bigr)$ and $\cO\bigl(n+n^{1/2}\epsilon^{-1}\bigr)$ 
for solving the expectation and finite-sum cases respectively.
\citet{SPIDER2018NeurIPS} introduced another variance reduction technique
called \textsc{Spider}, which can be viewed as a more general variant of SARAH.
They obtained sample complexities $\cO\bigl(\epsilon^{-3/2}\bigr)$ 
and $\cO\bigl(\min\{\epsilon^{-3/2},\,n^{1/2}\epsilon^{-1}\}\bigr)$ 
for the two cases respectively, but require small step sizes that are
proportional to~$\epsilon$.
\citet{SpiderBoost2018} extended \textsc{Spider} to obtain the same 
complexities with constant step sizes
and $\cO\bigl((n+\kappa^2)\log\epsilon^{-1}\bigr)$ under the gradient-dominant
condition.
In addition, \citet{NestedSVRG2018NeurIPS} obtained similar results using
a nested SVRG approach.


In addition to the above works on solving special cases of~\eqref{prob:main}
and~\eqref{prob:main-finite}, there are also considerable recent works on a
more general, two-layer stochastic composite optimization problem
\begin{equation}\label{prob:Exp-2-comp}
\minimize_{x\in\R^d}\quad  
\E_\nu\bigl[f_\nu\left(\E_\xi[g_\xi(x)]\right)\bigr] + r(x) \, ,
\end{equation}
where $f_\nu$ is parametrized by another random variables~$\nu$, 
which is independent of~$\xi$.
When $r\equiv 0$, \citet{SCGD-M.Wang} derived algorithms to find 
an $\epsilon$-approximate solution
with sample complexities $\cO(\epsilon^{-4})$, $\cO(\epsilon^{-3.5})$ and 
$\cO(\epsilon^{-1.25})$ for the smooth nonconvex case,
smooth convex case and smooth strongly convex case respectively.
For nontrivial convex~$r$, \citet{ASC-PG-M.Wang} obtained improved sample 
complexity of $\cO(\epsilon^{-2.25})$,
$\cO(\epsilon^{-2})$ and $\cO(\epsilon^{-1})$ 
for the three cases mentioned above respectively.

As a special case of~\eqref{prob:Exp-2-comp}, the following finite-sum 
problem also received significant attention:
\begin{equation}\label{prob:Finite-2-Comp}
\minimize_{x\in\R^d}\quad 
\frac{1}{m}\sum_{j=1}^m f_j\biggl(\frac{1}{n}\sum_{i=1}^n g_i(x)\biggr)+r(x)\,.
\end{equation}
When $r\equiv 0$ and the overall objective function is strongly convex,
\citet{SVR-SCGD} derived two algorithms based on the SVRG scheme to attain
sample complexities $\cO((m+ n + \kappa^3)\log\epsilon^{-1}))$ and 
$\cO((m+ n + \kappa^4)\log\epsilon^{-1}))$ respectively, 
where~$\kappa$ is some suitably defined condition number.
\citet{VRSC-PG} also used the SVRG scheme to obtain an  
$\cO(m+n+(m+n)^{2/3}\epsilon^{-1})$
complexity for the smooth nonconvex case and  
$\cO((m+ n + \kappa^3)\log\epsilon^{-1}))$ 
for strongly convex problems with nonsmooth~$r$. 
More recently, \citet{ZhangXiao2019C-SAGA} proposed a composite randomized
incremental gradient method based on the SAGA estimator 
\cite{defazio2014saga}, which matches the best known 
$\cO(m+n+(m+n)^{2/3}\epsilon^{-1})$ complexity
when~$F$ is smooth and nonconvex, and obtained an improved complexity
$\cO\bigl((m+n+\kappa(m+n)^{2/3})\log\epsilon^{-1}\bigr)$ under 
either gradient dominant or strongly convex assumptions.
When applied to the special cases~\eqref{prob:main} 
and~\eqref{prob:main-finite} we focus on in this paper ($m=1$), 
these results are strictly worse than ours in Table~\ref{tab:complexities}.

\subsection{Contributions and Outline} 

We develop the CIVR method by extending the variance reduction technique of
SARAH \citep{SARAH2017ICML,SmoothSARAH2019,ProxSARAH2019} 
and \textsc{Spider} \citep{SPIDER2018NeurIPS,SpiderBoost2018}
to solve the composite optimization problems~\eqref{prob:main} 
and~\eqref{prob:main-finite}.
The complexities of CIVR in Table~\ref{tab:complexities} match
the best results for solving the non-composite 
problems~\eqref{prob:exp-non-composite} and~\eqref{prob:erm-finite-sum},
despite the additional outer composition and the composite-gradient
estimator always being biased.
In addition:
\begin{itemize}
    \item
    It is shown in \citep{SPIDER2018NeurIPS} that the 
    $\cO\bigl(\min\{\epsilon^{-3/2},\,n^{1/2}\epsilon^{-1}\}\bigr)$ complexity
    is nearly optimal for the non-composite finite-sum 
    optimization problem~\eqref{prob:erm-finite-sum}.
    Therefore, we do not expect algorithms with better complexity for solving 
    the more general composite finite-sum problem~\eqref{prob:main-finite}.
    \item
    Under the assumptions of gradient dominance or strong convexity, the 
    $\cO\bigl(\bigl(n\!+\!\kappa n^{1/2}\bigr)\log\epsilon^{-1}\bigr)$
    complexity only appeared for the special case~\eqref{prob:erm-finite-sum}
    in the recent work~\citep{LiLi2018NeurIPS}.
\end{itemize}

Our results indicate that the additional smooth composition
in~\eqref{prob:main} and~\eqref{prob:main-finite}
does not incur higher complexity compared with~\eqref{prob:exp-non-composite}
and~\eqref{prob:erm-finite-sum}, 
despite the difficulty of dealing with biased estimators.
We believe these results can also be extended to the two-layer 
problems~\eqref{prob:Exp-2-comp} and~\eqref{prob:Finite-2-Comp},
by replacing~$n$ with $m+n$ in Table~\ref{tab:complexities}. 
But the extensions require quite different techniques and we will address them 
in a separate paper.

The rest of this paper is organized as follows. 
In Section~\ref{sec:algm}, we introduce the CIVR method.
In Section~\ref{sec:analysis}, we present convergence results of CIVR 
for solving the composite optimization problems~\eqref{prob:main} 
and~\eqref{prob:main-finite} and the required parameter settings.
Better complexities of CIVR under the gradient-dominant and optimally strongly 
convex conditions are given in Section~\ref{sec:fast-rates}.
In Section~\ref{sec:experiments}, we present numerical experiments for solving
a risk-averse portfolio optimization problem~\eqref{prob:risk-averse} 
on real-world datasets.

\section{The composite incremental variance reduction (CIVR) method}
\label{sec:algm}

\SetAlgoHangIndent{1.5em}
\setlength{\algomargin}{1em}
\begin{algorithm2e}[t]
    \DontPrintSemicolon
	\caption{Composite Incremental Variance Reduction (CIVR)}
	\label{alg:CIVR}
    \textbf{input:} 
    initial point $x^1_0$, step size $\eta>0$, number of epochs $T\geq1$,
    and a set of triples $\{\tau_t, B_t, S_t\}$ for $t=1,\ldots,T$,
    where $\tau_t$ is the epoch length and $B_t$ and $S_t$ are sample sizes
    in epoch~$t$. \;
	\For{$t = 1,...,T$}{
		Sample a set $\cB_t$ with size $B_t$ from the distribution of $\xi$,
        and construct the estimates  \\
        \vspace{-0.5ex}
		\begin{equation}
		\label{defn:sarah-1}
		y_0^t = \frac{1}{B_t}\sum_{\xi\in\cB_t}g_\xi(x_0^t), \qquad
		z_0^t =  \frac{1}{B_t}\sum_{\xi\in\cB_t}g_\xi'(x_0^t), \qquad
		\end{equation}
        Compute 
        $\tnf(x_0^t) = (z_0^t)^Tf'(y_0^t)$ and update:
        $x_1^t = \prox_r^{\eta}\bigl(x_0^t-\eta \tnf(x_0^t)\bigr)$. \;
        \vspace{0.5ex}
		\For{$i = 1,...,\tau_t-1$}{
        Sample a set $\cS_i^t$ with size $S_t$ from the distribution of $\xi$,
        and construct the estimates  \\[-2ex]
		\begin{eqnarray}
		\label{defn:sarah-2}
            y_i^t &=& y_{i-1}^t+ \frac{1}{S_t}\sum_{\xi\in\cS_i^t}\left(g_\xi(x_i^t)-g_\xi(x_{i-1}^t)\right), \\[-1ex]
            z_i^t &=& z_{i-1}^t \, +\, \frac{1}{S_t}\sum_{\xi\in\cS_i^t}\bigl(g'_\xi(x_i^t)-g'_\xi(x_{i-1}^t)\bigr).
		\end{eqnarray}
        Compute $\tnf(x_i^t) = (z_i^t)^Tf'(y_i^t)$ and update:
        $x_{i+1}^t = \prox_r^{\eta}\bigl(x_i^t-\eta \tnf(x_i^t)\bigr)$. \;
	}
    Set $x_0^{t+1} = x_{\tau_t}^t$. \;
	}
    \vspace{-1ex}
    \textbf{output:} $\bar{x}$ randomly chosen from  
    $\bigl\{x_i^t\bigr\}_{i=0,...,\tau_t-1}^{t = 1,...,T}$. \; 
\end{algorithm2e} 

With the notations in~\eqref{eqn:simple-notation}, we can write the composite
stochastic optimization problem~\eqref{prob:main} as
\begin{equation}\label{eqn:min-smooth-plus-simple}
    \minimize_{x\in\R^d}\quad \bigl\{ \Phi(x) = F(x) + r(x) \bigr\} \,,
\end{equation}
where $F$ is smooth and $r$ is convex.
The proximal operator of~$r$ with parameter~$\eta$ is defined as
\be
\label{defn:prox-operator}
\prox^{\eta}_r(x):=\argmin_y \,\Bigl\{r(y) + \frac{1}{2\eta}\|y-x\|^2\Bigr\}.
\ee
We assume that~$r$ is relatively simple, meaning that its 
proximal operator has a closed-form solution or can be computed efficiently.
The proximal gradient method 
\citep[e.g.,][]{Nesterov2013composite,Beck2017book}
for solving problem~\eqref{eqn:min-smooth-plus-simple} is
\begin{equation}\label{eqn:prox-grad-method}
    x^{t+1} = \prox_r^\eta\bigl(x^t-\eta F'(x^t)\bigr) \,,
\end{equation}
where $\eta$ is the step size. 
The \emph{proximal gradient mapping} of $\Phi$ is defined as
\be
\label{defn:ProxGD}
\mathcal{G}_\eta(x)  \triangleq
\frac{1}{\eta}\Bigl(x - \prox^{\eta}_r\bigl(x-\eta F'(x)\bigr)\Bigr).
\ee
As a result, the proximal gradient method~\eqref{eqn:prox-grad-method}
can be written as $x^{t+1} = x^{t} - \eta\, \mathcal{G}(x^t)$.
Notice that when $r\equiv 0$, $\prox_r^\eta(\cdot)$ becomes the identity
mapping and we have $\mathcal{G}_\eta(x)\equiv F'(x)$ for any $\eta>0$.

Suppose $\bar x$ is generated by a randomized algorithm.
We call $\bar x$ an $\epsilon$-stationary point in expectation if 
\begin{equation}\label{eqn:eps-stationary}
\E\bigl[\|\mathcal{G}_\eta(\bar x)\|^2\bigr] \leq \epsilon.
\end{equation}
(We assume that $\eta$ is a constant that does not depend on~$\epsilon$.)
As we mentioned in the introduction, we measure the efficiency of an algorithm
by its sample complexity of $g_\xi$ and their Jacobian $g'_\xi$, i.e., 
the total number of times they need to be evaluated, in order to find 
a point~$\bar x$ that satisfies~\eqref{eqn:eps-stationary}.
Our goal is to develop a randomized algorithm that has low sample 
complexity.

We present in Algorithm~\ref{alg:CIVR} the Composite Incremental Variance
Reduction (CIVR) method.
This methods employs a two time-scale variance-reduced estimator
for both the inner function value of $g(\cdot)=\E_\xi[g_\xi(\cdot)]$ 
and its Jacobian $g'(\cdot)$.
At the beginning of each outer iteration~$t$ (each called an epoch), 
we construct a relatively accurate estimate $y_0^t$ for $g(x_0^t)$ and 
$z_0^t$ for $g'(x_0^t)$ respectively, using a relatively large sample
size $B_t$. 
During each inner iteration~$i$ of the $t$th epoch, we construct an estimate
$y_i^t$ for $g(x_i^t)$ and $z_i^t$ for $g'(x_i^t)$ respectively, 
using a smaller sample size $S_t$ \emph{and} incremental corrections from the
previous iterations. 
Note that the epoch length $\tau_t$ and the sample sizes $B_t$ and $S_t$ 
are all adjustable for each epoch~$t$. Therefore, besides setting a constant
set of parameters, we can also adjust them gradually in order to obtain
better theoretical properties and practical performance.

This variance-reduction technique was first proposed as part of SARAH 
\cite{SARAH2017ICML} where it is called \emph{recursive} variance reduction.
It was also proposed in \citep{SPIDER2018NeurIPS} in the form of a
\emph{Stochastic Path-Integrated Differential EstimatoR} (\textsc{Spider}).
Here we simply call it \emph{incremental} variance reduction.
A distinct feature of this incremental estimator is that the inner-loop 
estimates $y_i^t$ and $z_i^t$ are biased, i.e.,
\begin{equation}\label{eqn:biased-estimate}
\begin{cases}
    ~\E[y_i^t|x_i^t] = g(x_i^t) - g(x_{i-1}^t) + y_{i-1}^t \neq g(x_i^t)\,, \\
    ~\E[z_i^t|x_i^t] = g'(x_i^t) - g'(x_{i-1}^t) + z_{i-1}^t \neq g'(x_i^t)\,.
\end{cases} 
\end{equation}
This is in contrast to two other popular variance-reduction techniques, 
SVRG \citep{johnson2013accelerating} and SAGA \citep{defazio2014saga},
whose gradient estimators are always unbiased.
Note that unbiased estimators for $g(x_i^t)$ and $g'(x_i^t)$ are not
essential here, because the composite estimator 
$\tilde{\nabla}F(x_i^t)=(z_i^t)^T f'(y_i^t)$
is always biased.

\section{Convergence Analysis}
\label{sec:analysis}

In this section, we present theoretical results on the convergence properties
of CIVR (Algorithm~\ref{alg:CIVR}) when the composite function~$F$ is smooth.
More specifically, we make the following assumptions.

\begin{assumption}
\label{assumption:Lip}
The following conditions hold concerning problems~\eqref{prob:main}
and~\eqref{prob:main-finite}:
\begin{itemize} \itemsep 0pt
    \item 
	$f: \R^{p}\rightarrow\R$ is a $C^1$ smooth and $\ell_f$-Lipschitz function and its gradient $f'$ is $L_f$-Lipschitz. 
    \item Each $g_\xi:\R^{d}\rightarrow\R^p$ is a $C^1$ smooth and $\ell_g$-Lipschitz vector mapping and its Jacobian $g_\xi'$ is $L_g$-Lipschtiz. 
    Consequently, $g$ in~\eqref{eqn:simple-notation}
    is $\ell_g$-Lipschitz and its Jacobian $g'$ is $L_g$-Lipschitz.
    \item $r:\R^{d}\rightarrow\R\cup\{\infty\}$ is a convex and lower-semicontinuous function. 
    \item The overall objective function $\Phi$ is bounded below, i.e.,
        $\Phi^*=\inf_x\Phi(x)>-\infty$.
\end{itemize}
\end{assumption}
\vspace{0.5ex}
\begin{assumption}
	\label{assumption:variance}
    For problem~\eqref{prob:main}, we further assume that
    there exist constants $\sigma_g$ and $\sigma_{g'}$ such that
	\begin{equation}
	\E_\xi[\|g_\xi(x) - g(x)\|^2]\leq \sigma_g^2 \, , \qquad 
    \E_\xi[\|g'_\xi(x) - g'(x)\|^2]\leq \sigma_{g'}^2 \, .
	\end{equation}
\end{assumption}

As a result of Assumption~\ref{assumption:Lip},
$F(x)=f\bigl(g(x)\bigr)$ is smooth and $F'$ is $L_F$-Lipschitz continuous with
\[
    L_F = \ell_g^2 L_f + \ell_f L_g
\]
(see proof in the supplementary materials).
For convenience, we also define two constants
\begin{equation}
\label{defn:G0-sigma0}
G_0: = 2\bigl(\ell_g^4 L_f^2  + \ell_f^2 L_g^2\bigr) \, ,
\qquad \mbox{and}\qquad 
\sigma_0^2 := 2\bigl(\ell_g^2 L_f^2\sigma_{g}^2 +\ell_f^2\sigma_{g'}^2\bigr)\,.
\end{equation}
It is important to notice that $G_0=\cO(L_F^2)$, since we will use step size
$\eta=\Theta(1/\sqrt{G_0})=\Theta(1/L_F)$.

In the next two subsections, we present complexity analysis of CIVR for
solving problem~\eqref{prob:main} and~\eqref{prob:main-finite} respectively. 
Due to the space limitation, 
all proofs are provided in the supplementary materials.

\subsection{The composite expectation case}

The following results for solving problem~\eqref{prob:main} are presented
with notations defined in~\eqref{eqn:simple-notation}, 
\eqref{defn:ProxGD} and~\eqref{defn:G0-sigma0}.

\begin{theorem}
	\label{theorem:cvg-general-fixed} 
    Suppose Assumptions~\ref{assumption:Lip} and~\ref{assumption:variance} hold.
	Given any $\epsilon>0$, 
    we set $T = \lceil 1/\sqrt{\epsilon} \rceil$ and
\[
    \tau_t=\tau=\lceil 1/\sqrt{\epsilon} \rceil, \quad
    B_t = B = \lceil \sigma_0^2/\epsilon \rceil, \quad
    S_t = S = \lceil 1/\sqrt{\epsilon} \rceil, \quad
    \mbox{for}\quad t=1,\ldots,T.
\]
    Then as long as $\eta\leq\frac{4}{L_F + \sqrt{L_F^2 + 12G_0}}$, 
    the output $\bar{x}$ of Algorithm~\ref{alg:CIVR} satisfies
	\begin{equation}
	\label{thm:cvg-general-fixed}
	\E\bigl[\|\mathcal{G}_\eta(\bar{x})\|^2\bigr] 
    \leq 
    \Bigl(8\bigl(\Phi(x_0^1) - \Phi^*\bigr)\eta^{-1} + 6\Bigr)\cdot\epsilon 
    = \cO(\epsilon).
	\end{equation}
    As a result, the sample complexity of obtaining an $\epsilon$-approximate
    solution is $TB + 2 T\tau S = \cO\bigl(\epsilon^{-3/2}\bigr)$.
\end{theorem}

Note that in the above scheme, the epoch lengths $\tau_t$ and
all the batch sizes $B_t$ and $S_t$ are set to be constant (depending on a
pre-fixed $\epsilon$) without regard of~$t$. 
Intuitively, we do not need as many samples in the early stage of the algorithm
as in the later stage.
In addition, it will be useful in practice to have a variant of the algorithm 
that can adaptively choose $\tau_t$, $B_t$ and $S_t$ throughout the epochs
without dependence on a pre-fixed precision.
This is done in the following theorem. 

\vspace{0.5ex}

\begin{theorem}
	\label{theorem:cvg-general-adaptive} 
    Suppose Assumptions~\ref{assumption:Lip} and~\ref{assumption:variance} hold.
    We set $\tau_t = S_t = \lceil at + b\rceil$ and 
    $B_t=\lceil \sigma_0^2(at+b)^2 \rceil$ where $a>0$ and $b\geq0$. 
    Then as long as $\eta\leq\frac{4}{L_F + \sqrt{L_F^2 + 12G_0}}$, we have 
    for any $T\geq 1$,
	\begin{equation}
	\label{thm:cvg-general-adaptive}
	\E\bigl[\|\mathcal{G}_\eta(\bar{x})\|^2\bigr] 
    \leq  \frac{2}{aT^2+(a+2b)T}\left(\frac{8\bigl(\Phi(x_0^1) - \Phi^*\bigr)}{\eta} + \frac{6}{a+b}+\frac{6}{a}\ln\left(\frac{aT+b}{a+b}\right)\right)
	 = \cO\Bigl(\frac{\ln T}{T^2}\Bigr).
	\end{equation}
    As a result, obtaining an $\epsilon$-approximate solution requires $T = \tilde{\cO}(1/\sqrt{\epsilon})$ epochs and a total sample complexity of $\tilde{\cO}\bigl(\epsilon^{-3/2}\bigr)$, where the $\tilde{\cO}(\cdot)$ notation hides logarithmic factors.
\end{theorem}

\subsection{The composite finite-sum case}

In this section, we consider the composite finite-sum optimization
problem~\eqref{prob:main-finite}.
In this case, the random variable $\xi$ has a uniform distribution over
the finite index set $\{1,...,n\}$.
At the beginning of each epoch in Algorithm~\ref{alg:CIVR}, 
we use the full sample size $\cB_t=\{1,\ldots,n\}$
to compute $y_0^t$ and $z_0^t$.
Therefore $B_t=n$ for all $t$ and 
Equation~\eqref{defn:sarah-1} in Algorithm~\ref{alg:CIVR} becomes
\begin{equation}
\label{defn:sarah-1-finte}
y_0^t= g(x_0^t) = \frac{1}{n}\sum_{j=1}^ng_j(x_0^t) \,,
\qquad 
z_0^t= g'(x_0^t) = \frac{1}{n}\sum_{j=1}^ng'_j(x_0^t) \,.
\end{equation}
Also in this case, we no longer need Assumption~\ref{assumption:variance}.

\begin{theorem}
	\label{theorem:cvg-finite-fixed} 
    Suppose Assumptions~\ref{assumption:Lip} holds.
    Let the parameters in Algorithm~\ref{alg:CIVR} be set as
    $\cB_t=\{1,\ldots,n\}$ and 
    $\tau_t = S_t = \lceil \sqrt{n} \rceil$ for all~$t$. 
    Then as long as $\eta\leq\frac{4}{L_F + \sqrt{L_F^2 + 12G_0}}$, 
    we have for any $T\geq 1$, 
	\begin{equation}
	\label{thm:cvg-finite-fixed}
	\E\bigl[\|\mathcal{G}_\eta(\bar{x})\|^2\bigr] \leq \frac{8\bigl(\Phi(x_0^1) - \Phi^*\bigr)}{\eta\sqrt{n}T} = \cO\left(\frac{1}{\sqrt{n}T}\right),
	\end{equation}
    As a result, obtaining an $\epsilon$-approximate solution requires 
$T = \cO\bigl(1/(\sqrt{n}\epsilon)\bigr)$ epochs and a total sample complexity of $TB + 2 T\tau S = \cO\bigl(n+\sqrt{n}\epsilon^{-1}\bigr)$.
\end{theorem}

Similar to the previous section, we can also choose the epoch lengths and
sample sizes adaptively to save the sampling cost in the early stage of 
the algorithm. 
However, due to the finite-sum structure of the problem, when the batch size $B_t$ reaches $n$, we will start to take the full batch at the beginning of each
epoch to get the exact $g(x_0^t)$ and $g'(x_0^t)$. 
This leads to the following theorem.

\begin{theorem}
	\label{theorem:cvg-finite-adaptive} 
    Suppose Assumptions~\ref{assumption:Lip} holds.
    For some positive constants $a>0$ and $0\leq b<\sqrt{n}$, 
    denote $T_0 := \bigl\lceil\frac{\sqrt{n}-b}{a}\bigr\rceil = \cO\bigl(\sqrt{n}\bigr)$.
    When $t\leq T_0$ we set the parameters to be $\tau_t = S_t = \sqrt{B_t} = \lceil at + b\rceil$; when $t>T_0$, we set $\cB_t=\{1,\ldots,n\}$ and $\tau_t = S_t = \bigl\lceil\sqrt{n}\bigr\rceil$.
    Then as long as $\eta\leq\frac{4}{L_F + \sqrt{L_F^2 + 12G_0}}$, 
    \vspace{-1ex}
	\begin{eqnarray}
	\label{thm:cvg-finite-adaptive}
	\E\bigl[\|\mathcal{G}_\eta(\bar{x})\|^2\bigr]  \leq \begin{cases}
	\cO\bigl(\frac{\ln T}{T^2}\bigr) &\mbox{ if } T\leq T_0 \, ,\\
    \cO\bigl(\frac{\ln n}{\sqrt{n}(T-T_0+1)}\bigr) &\mbox{ if } T>T_0 \, .
	\end{cases}
	\end{eqnarray}
	As a result, 
    the total sample complexity of Algorithm~\ref{alg:CIVR} for
    obtaining an $\epsilon$-approximate solution is
    $\tilde{\cO}\bigl(\min\bigl\{\sqrt{n}\epsilon^{-1},\epsilon^{-3/2}\bigr\}\bigr)$, where $\tilde{\cO}(\cdot)$ hides logarithmic factors.
\end{theorem}

\section{Fast convergence rates under stronger conditions}
\label{sec:fast-rates}

\vspace{-1ex}

In this section we consider two cases where fast linear convergence can be guaranteed for CIVR. 

\vspace{-1ex}

\subsection{Gradient-dominant function}
The first case is when $r \equiv 0$ and 
$F$ is \emph{$\nu$-gradient dominant}, i.e., there is some $\nu>0$ such that
\begin{equation}
\label{defn:Grd-domi}
F(x) - \inf_yF(y)\leq \frac{\nu}{2}\|F'(x)\|^2, \qquad \forall\, x\in\R^d.
\end{equation}
Note that a $\mu$-strongly convex function is $(1/\mu)$-gradient dominant by this definition. Hence strong convexity is a special case of the gradient dominant condition, which in turn is a special case of the Polyak-{\L}ojasiewicz condition with the {\L}ojasiewicz exponent equal to 2
\citep[see, e.g.,][]{KarimiNutiniSchmidt2016}.

In order to solve~\eqref{prob:main} with a pre-fixed precision~$\epsilon$,
we use a periodic restart strategy depicted below.
\begin{theorem}
	\label{theorem:cvg-KL-general} 
    Consider~\eqref{prob:main} with $r\equiv 0$. Suppose 
    Assumptions~\ref{assumption:Lip} and~\ref{assumption:variance} hold
    and~$F$ is $\nu$-gradient dominant.
    Given any $\epsilon>0$, let 
    $\tau_t = S_t = \bigl\lceil\frac{1}{\sqrt{\epsilon}}\bigr\rceil$,
    $B_t=\bigl\lceil \frac{12\nu\sigma_0^2}{\epsilon} \bigr \rceil$ and 
    $T = \bigl\lceil\frac{16\nu\sqrt{\epsilon}}{\eta}\bigr\rceil$. 
    Then as long as $\eta\leq\frac{4}{L_F + \sqrt{L_F^2 + 12G_0}}$, 
    \vspace{-1ex}
	\begin{equation}
	\label{thm:cvg-KL-general-0}
    \E\bigl[F(\bar{x}) - F^*\bigr] \leq \frac{1}{2}\bigl(F(x_0^1) - F^*\bigr) + \frac{1}{2}\epsilon.
	\end{equation}
    Therefore, we can periodically restart Algorithm~\ref{alg:CIVR} after 
    every~$T$ epochs 
    (using the output of previous period as input to the new period), 
    then $\E[F(\bar{x}) - F^*]$ converges linearly to 
    $\epsilon$ with a factor of $\frac{1}{2}$ per period.
    As a result, the sample complexity for finding an $\epsilon$-solution is
    $\cO\bigl(\bigl(\nu\epsilon^{-1}\bigr)\ln \epsilon^{-1}\bigr)$.
\end{theorem}

The restart strategy also applies to the finite-sum case. 

\begin{theorem}
	\label{theorem:cvg-KL-finite} 
    Consider problem~\eqref{prob:main-finite} with $r\equiv 0$.
    Suppose Assumption~\ref{assumption:Lip} hold and $F$ is $\nu$-gradient dominant.
    If we set $\tau_t = S_t = \sqrt{B_t} = \lceil\sqrt{n}\rceil$ and
    $T = \bigl\lceil\frac{16\nu}{\sqrt{n}\eta}\bigr\rceil$,
    then as long as $\eta\leq\frac{4}{L_F + \sqrt{L_F^2 + 12G_0}}$, 
    \vspace{-1ex}
	\begin{equation}
	\label{thm:cvg-KL-finite-0}
	\E\bigl[F(\bar{x}) - F^*\bigr] \leq \frac{1}{2}\bigl(F(x_0^1) - F^*\bigr)\,.
	\end{equation}
	By periodically restart Algorithm~\ref{alg:CIVR} after every~$T$ epochs, 
    $\E[F(\bar{x}) - F^*]$ converges linearly to $0$. 
    As a result, the sample complexity for finding an $\epsilon$-solution is
    $\cO\bigl(\bigl(n+\frac{\nu\sqrt{n}}{\eta}\bigr)\ln\frac{1}{\epsilon}\bigr)$.
\end{theorem}


\subsection{Optimally strongly convex function}

In this part, we assume a \emph{$\mu$-optimally strongly convex} condition 
on the function $\Phi(x)=F(x)+r(x)$, i.e., there exists a $\mu>0$ such that
\vspace{-0.5ex}
\be
\label{defn:O-S-CVX}
\Phi(x) - \Phi(x^*)\geq \frac{\mu}{2}\|x-x^*\|^2, \qquad \forall x\in\R^d .
\ee 
We have the following two results for solving problems~\eqref{prob:main}
and~\eqref{prob:main-finite} respectively.
\begin{theorem}
	\label{theorem:cvg-O-S-CVX-general}
	Consider problem~\eqref{prob:main}. 
    Suppose Assumptions~\ref{assumption:Lip} and~\ref{assumption:variance} hold
    and $\Phi$ is $\mu$-optimally strongly convex.
    We set $\tau_t = S_t = \bigl\lceil \frac{1}{\sqrt{\epsilon}} \bigr\rceil$,
    $B_t=\bigl\lceil \frac{9\sigma_0^2}{2\mu\epsilon} \bigr\rceil$
    and $T = \lceil \frac{5\sqrt{\epsilon}}{\mu\eta} \rceil$. 
    Then if we choose $\eta<\frac{2}{L_F + \sqrt{L_F^2 + 36G_0}}$,
    \vspace{-1ex}
	\begin{equation}
	\label{thm:cvg-O-S-CVX-general-0}
	\E\bigl[\Phi(\bar{x}) - \Phi^*\bigr] \leq \frac{1}{2}\bigl(\Phi(x_0^1) - \Phi^*\bigr) + \frac{1}{2}\epsilon.
	\end{equation}
	By periodically restart Algorithm~\ref{alg:CIVR} after every~$T$ epochs, 
    $\E[\Phi(\bar{x}) - \Phi^*]$ converges linearly to $\epsilon$
    As a result, the sample complexity for finding an $\epsilon$-solution is
    $\cO\bigl(\mu^{-1}\epsilon^{-1}\ln \epsilon^{-1}\bigr)$.
\end{theorem}

\begin{theorem}
	\label{theorem:cvg-O-S-CVX-finite}
	Consider the finite-sum problem~\eqref{prob:main-finite}. 
    Suppose Assumption~\ref{assumption:Lip} hold and $\Phi$ is $\mu$-optimally
    strongly convex.
    We set $\tau_t = S_t = \sqrt{B_t} = \lceil\sqrt{n}\rceil$ 
    and $T = \bigl \lceil \frac{5}{\sqrt{n}\mu\eta} \bigr\rceil$.
    If $\eta<\frac{2}{L_F + \sqrt{L_F^2 + 36G_0}}$, then
    \vspace{-1ex}
	\begin{equation}
	\label{thm:cvg-O-S-CVX-finite-0}
	\E\bigl[\Phi(\bar{x}) - \Phi^*\bigr] \leq \frac{1}{2}\bigl(\Phi(x_0^1) - \Phi^*\bigr).
	\end{equation}
	By periodically restart Algorithm~\ref{alg:CIVR} after every $T$ epochs, 
    $\E[\Phi(\bar{x}) - \Phi^*]$ converges linearly to $0$ with rate $\frac{1}{2}$. 
    Therefore, the sample complexity of finding an $\epsilon$-solution is
    $\cO\bigl(\bigl(n+\frac{\sqrt{n}}{\mu\eta}\bigr)\ln \frac{1}{\epsilon}\bigr)$.
\end{theorem}

If we define a condition number $\kappa=L_F/\mu$, then
since $\eta=\Theta(1/L_F)$, we have $1/(\mu\eta)=O(\kappa)$ and the above
complexities become
$\cO\bigl(\bigl(\kappa\epsilon^{-1}\bigr)\ln\epsilon^{-1}\bigr)$ 
and $\cO\bigl(\bigl(n+\kappa n^{1/2}\bigr)\ln\epsilon^{-1}\bigr)$.

\section{Numerical Experiments}
\label{sec:experiments}

\begin{figure*}[t]
	\centering 
	\includegraphics[width=0.325\linewidth]{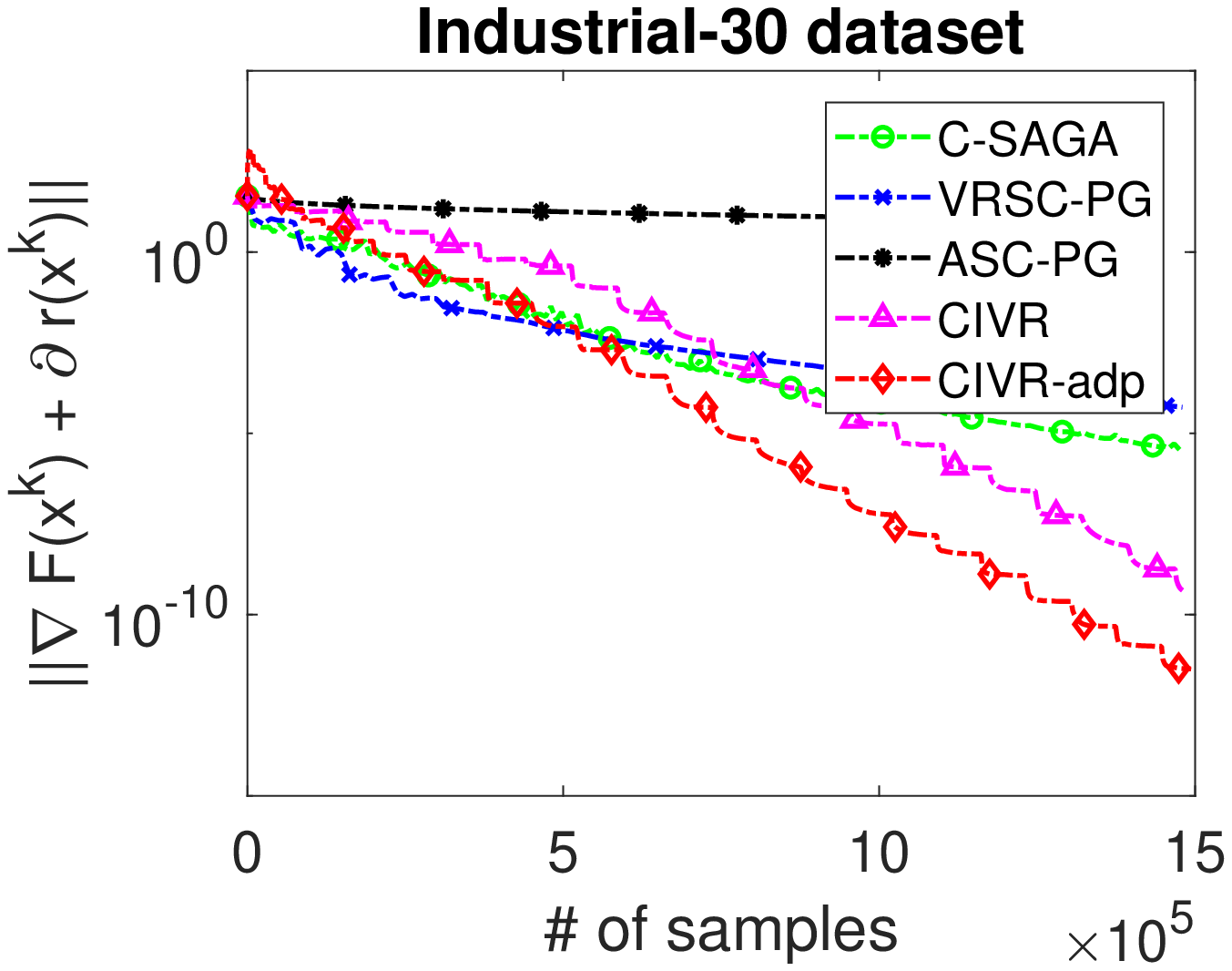}
	\includegraphics[width=0.325\linewidth]{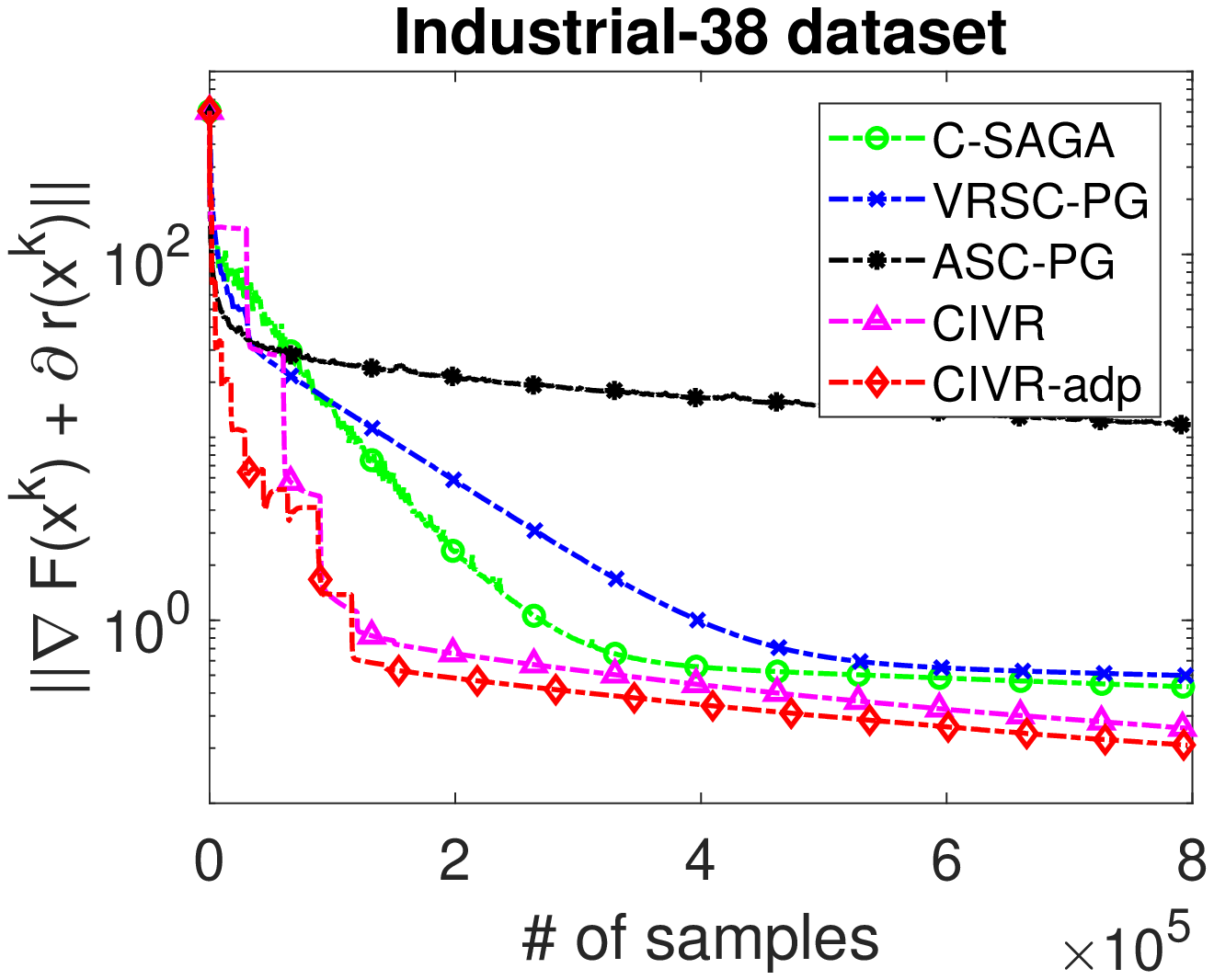}
    \includegraphics[width=0.325\linewidth]{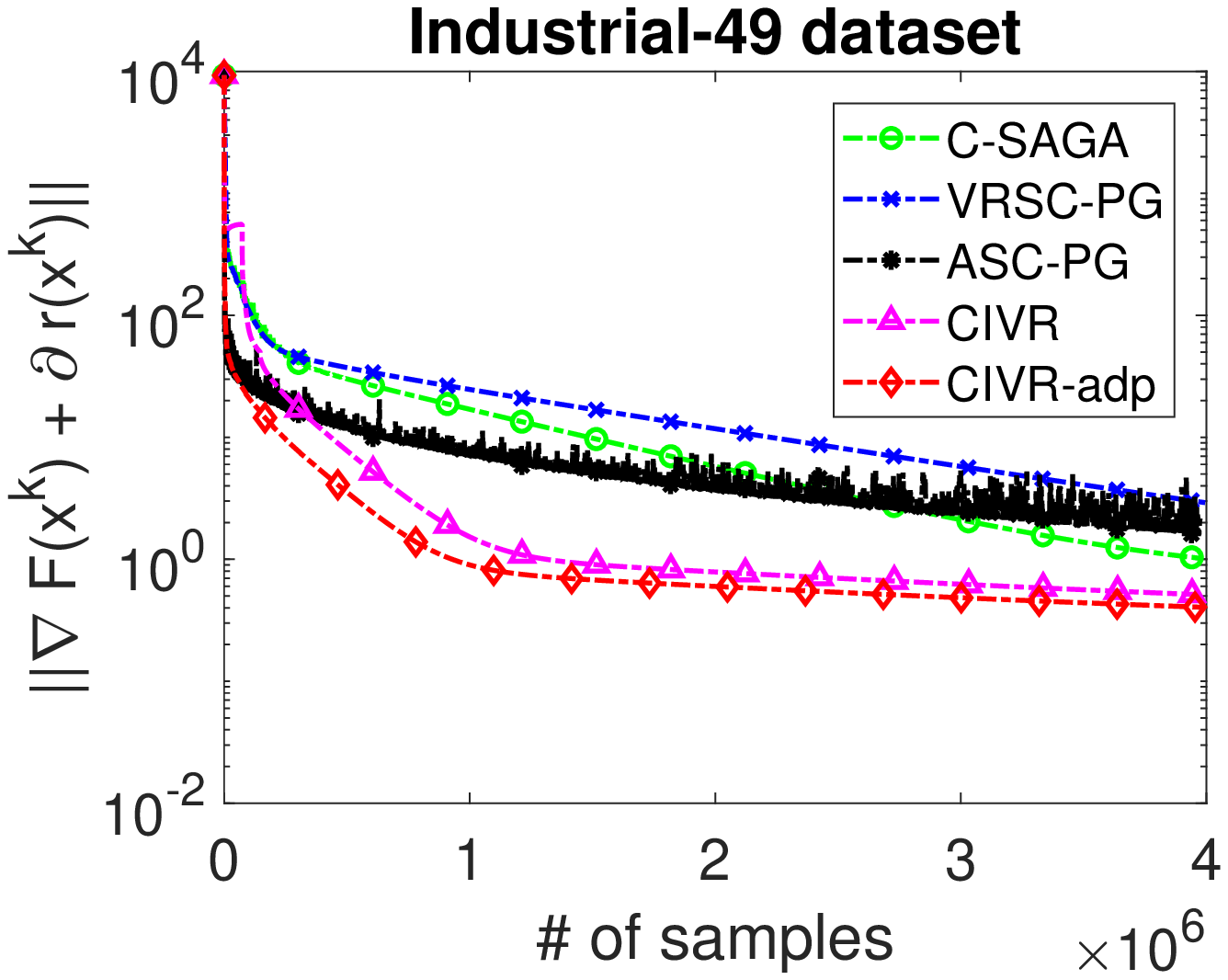}\\[1ex]
	\vfill  
	\includegraphics[width=0.325\linewidth]{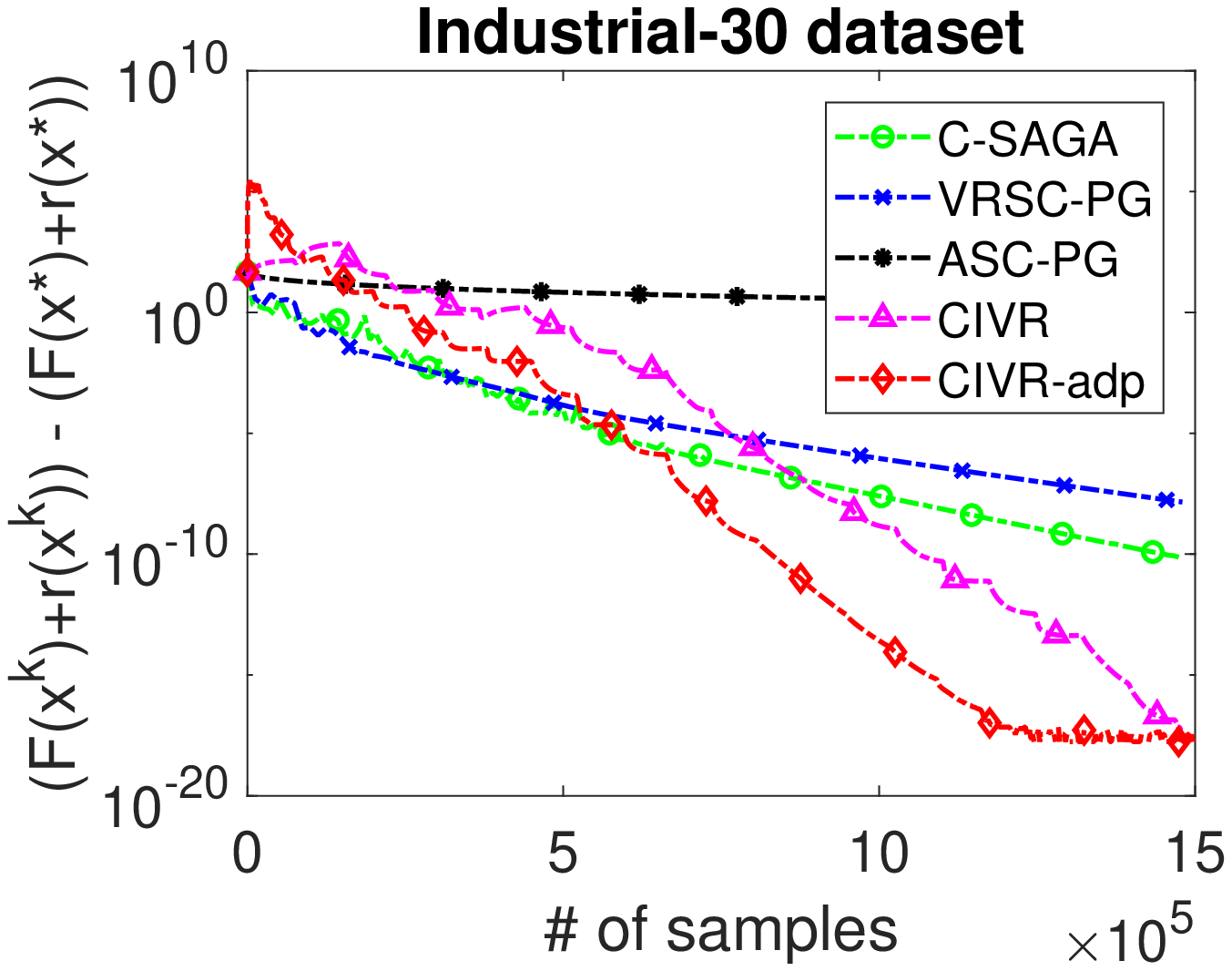}
    \includegraphics[width=0.325\linewidth]{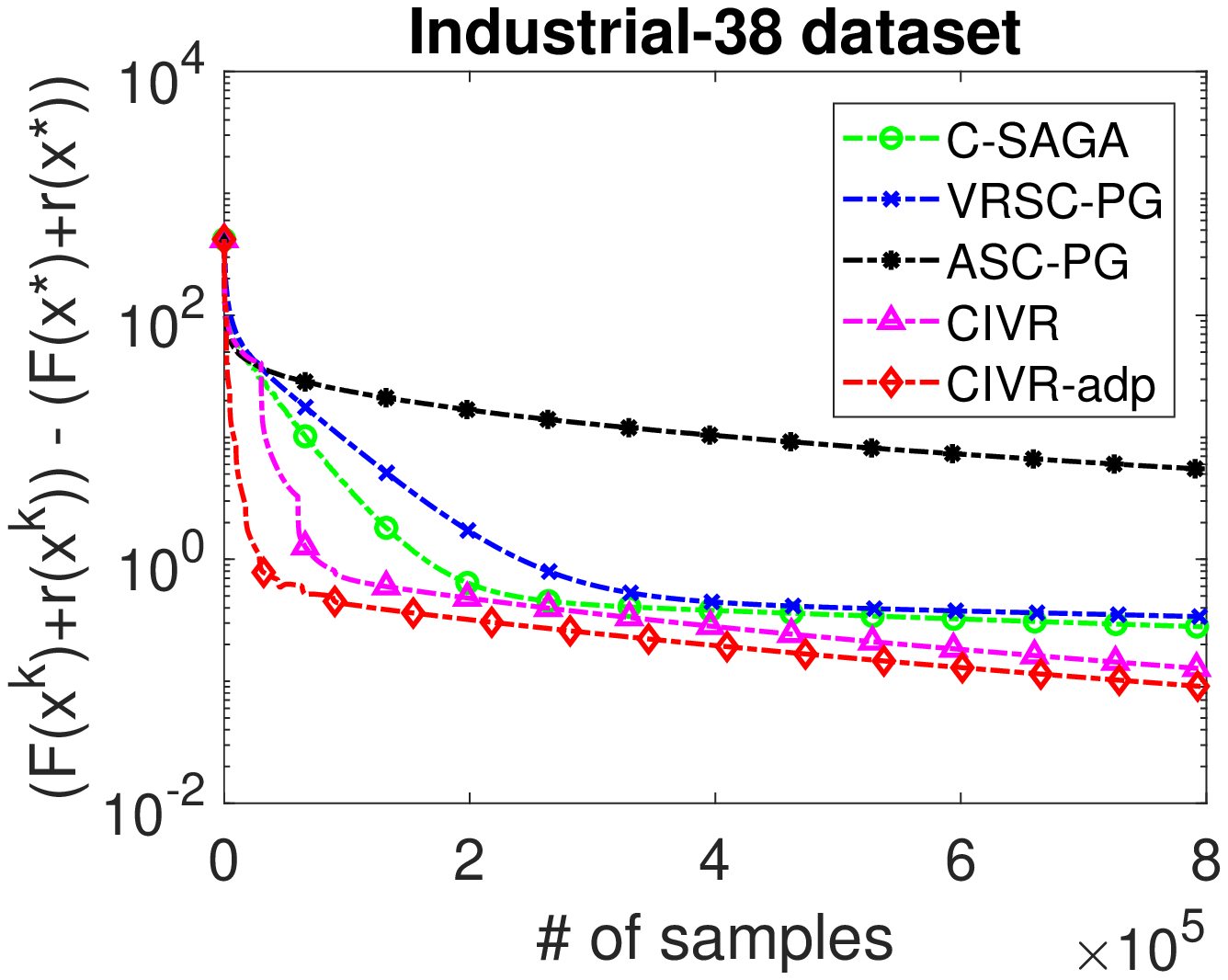}
	\includegraphics[width=0.325\linewidth]{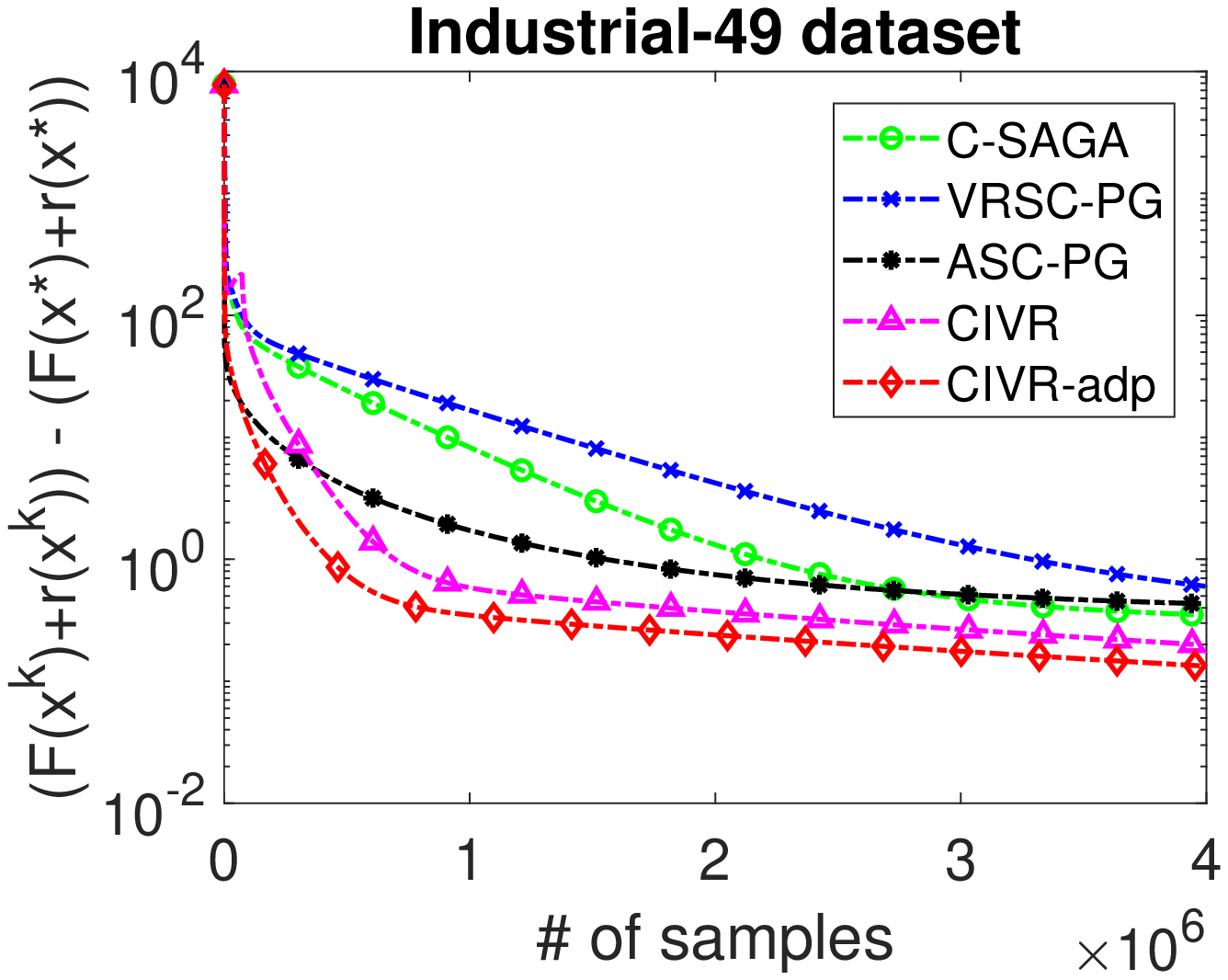}
    \caption{Experiments on the risk-averse portfolio optimization problem.}
	\label{fig:risk-averse}
\end{figure*}

In this section, we present numerical experiments for a risk-averse 
portfolio optimization problem. 
%
Suppose there are $d$ assets that one can invest during~$n$ time periods 
labeled as $\{1,...,n\}$. 
Let $R_{i,j}$ be the return or payoff per unit of asset~$j$ at time~$i$, 
and $R_i$ be the vector consists of $R_{i,1},\ldots,R_{i,d}$. 
Let $x\in\R^d$ be the decision variable, where each component $x_j$ represent
the amount of investment or percentage of the total investment allocated
to asset~$j$, for $j=1,\ldots,d$.
The same allocations or percentages of allocations are repeated over 
the~$n$ time periods. We would like to maximize the average return over
the~$n$ periods, but with a penalty on the variance of the returns 
across the~$n$ periods 
(in other words, we would like different periods to have similar returns). 

This problem can be formulated as a finite-sum version of
problem~\eqref{prob:risk-averse}, 
with a discrete random variable $\xi\!\in\!\{1,\ldots,n\}$ 
and $h_i(x)\!=\!\langle R_i, x\rangle$ for $i=1,\ldots,n$.
The function~$r$ can be chosen as the indicator function
of an $\ell_1$ ball, or a soft $\ell_1$ regularization term. 
We choose the later one in our experiments to obtain a sparse asset allocation. 
Using the mappings defined in~\eqref{eqn:mean-var-mapping}, it can be further
transformed into the composite finite-sum problem~\eqref{prob:main-finite},
hence readily solved by the CIVR method.
For comparison, 
we implement the C-SAGA algorithm \cite{ZhangXiao2019C-SAGA} as a benchmark. 
As another benchmark, this problem can also be formulated as a two-layer
composite finite-sum problem~\eqref{prob:Finite-2-Comp}, which was done 
in \citep{VRSC-PG} and \citep{SVR-SCGD}. We solve the two-layer formulation
by ASC-PG \citep{ASC-PG-M.Wang} and VRSC-PG \citep{VRSC-PG}.  
Finally, we also implemented CIVR-adp, which is the adaptive sampling variant
described in Theorem~\ref{theorem:cvg-finite-adaptive}.

We test these algorithms on three real world portfolio datasets, which contain 30, 38 and 49 industrial portfolios respectively, 
from the Keneth R. French Data Library\footnote{http://mba.tuck.dartmouth.edu/pages/faculty/ken.french/data\_library.html}.  
For the three datasets, the daily data of the most recent 24452, 10000 and 24400 days are extracted respectively to conduct the experiments. 
We set the parameter $\lambda = 0.2$ in~\eqref{prob:risk-averse}
and use an $\ell_1$ regularization $r(x) = 0.01\|x\|_1$. 
The experiment results are shown in Figure~\ref{fig:risk-averse}.
The curves are averaged over 20 runs and are plotted against the number of samples of the component functions (the horizontal axis). 

Throughout the experiments, VRSC-PG and C-SAGA algorithms use the batch size 
$S = \lceil n^{2/3}\rceil$ while CIVR uses the batch size 
$S = \lceil \sqrt{n}\rceil$, all dictated by their complexity theory.
CIVR-adp employs the adaptive batch size 
$S_t = \bigl\lceil\min\{\sqrt{10t + 1},\sqrt{n}\}\bigr\rceil$ for $t = 1,...,T$.
For Industrial-30 dataset, all of VRSC-PG, C-SAGA, CIVR and CIVR-adp use the 
same step size $\eta = 0.1$. 
They are chosen from the set $\eta\in\{1, 0.1, 0.01, 0.001, 0.0001\}$ 
by experiments.  
And $\eta = 0.1$ works best for all four tested methods simultaneously. 
Similarly, $\eta = 0.001$ is chosen for the Industrial-38 dataset 
and $\eta = 0.0001$ is chosen for the Industrial-49 dataset.
For ASC-PG, we set its step size parameters $\alpha_k = 0.001/k$ and
$\beta_k = 1/k$ \citep[see details in][]{ASC-PG-M.Wang}. 
They are hand-tuned to ensure ASC-PG converges fast among a range of tested parameters. 
Overall, CIVR and CIVR-adp outperform other methods.

\bibliographystyle{plainnat}
\bibliography{civr}

\begin{thebibliography}{35}
\providecommand{\natexlab}[1]{#1}
\providecommand{\url}[1]{\texttt{#1}}
\expandafter\ifx\csname urlstyle\endcsname\relax
  \providecommand{\doi}[1]{doi: #1}\else
  \providecommand{\doi}{doi: \begingroup \urlstyle{rm}\Url}\fi

\bibitem[Allen-Zhu(2017)]{Natasha2017ICML}
Zeyuan Allen-Zhu.
\newblock Natasha: Faster non-convex stochastic optimization via strongly
  non-convex parameter.
\newblock In \emph{Proceedings of the 34th International Conference on Machine
  Learning (ICML)}, volume~70 of \emph{Proceedings of Machine Learning
  Research}, pages 89--97, Sydney, Australia, 2017.

\bibitem[Allen-Zhu(2018)]{Natasha2NIPS2018}
Zeyuan Allen-Zhu.
\newblock Natasha 2: Faster non-convex optimization than {SGD}.
\newblock In \emph{Advances in Neural Information Processing Systems 31}, pages
  2675--2686. Curran Associates, Inc., 2018.

\bibitem[Allen-Zhu and Hazan(2016)]{Allen-ZhuHazan2016}
Zeyuan Allen-Zhu and Elad Hazan.
\newblock Variance reduction for faster non-convex optimization.
\newblock In \emph{Proceedings of the 33rd International Conference on Machine
  Learning (ICML)}, pages 699--707, 2016.

\bibitem[Beck(2017)]{Beck2017book}
Amir Beck.
\newblock \emph{First-Order Methods in Optimization}.
\newblock MOS-SIAM Series on Optimization. SIAM, 2017.

\bibitem[Chaudhari et~al.(2016)Chaudhari, Choromanska, Soatto, LeCun, Baldassi,
  Borgs, Chayes, Sagun, and Zecchina]{Pratik}
Pratik Chaudhari, Anna Choromanska, Stefano Soatto, Yann LeCun, Carlo Baldassi,
  Christian Borgs, Jennifer Chayes, Levent Sagun, and Riccardo Zecchina.
\newblock Entropy-sgd: Biasing gradient descent into wide valleys.
\newblock \emph{arXiv preprint, arXiv:1611.01838}, 2016.

\bibitem[Dann et~al.(2014)Dann, Neumann, and Peters]{dann2014policy}
Christoph Dann, Gerhard Neumann, and Jan Peters.
\newblock Policy evaluation with temporal differences: a survey and comparison.
\newblock \emph{Journal of Machine Learning Research}, 15\penalty0
  (1):\penalty0 809--883, 2014.

\bibitem[Defazio et~al.(2014)Defazio, Bach, and
  Lacoste-Julien]{defazio2014saga}
Aaron Defazio, Francis Bach, and Simon Lacoste-Julien.
\newblock {SAGA}: A fast incremental gradient method with support for
  non-strongly convex composite objectives.
\newblock In \emph{Advances in Neural Information Processing Systems 27}, pages
  1646--1654, 2014.

\bibitem[Fang et~al.(2018)Fang, Li, Lin, and Zhang]{SPIDER2018NeurIPS}
Cong Fang, Chris~Junchi Li, Zhouchen Lin, and Tong Zhang.
\newblock Spider: Near-optimal non-convex optimization via stochastic
  path-integrated differential estimator.
\newblock In \emph{Advances in Neural Information Processing Systems 31}, pages
  689--699. Curran Associates, Inc., 2018.

\bibitem[Ghadimi and Lan(2013)]{GhadimiLan2013}
Saeed Ghadimi and Guanghui Lan.
\newblock Stochastic first- and zeroth-order methods for nonconvex stochastic
  programming.
\newblock \emph{SIAM Journal on Optimization}, 23\penalty0 (4):\penalty0
  2341--2368, 2013.

\bibitem[Gulcehre et~al.(2016)Gulcehre, Moczulski, Visin, and
  Bengio]{Smoothing-2}
Caglar Gulcehre, Marcin Moczulski, Francesco Visin, and Yoshua Bengio.
\newblock Mollifying networks.
\newblock \emph{arXiv preprint, arXiv:1608.04980}, 2016.

\bibitem[Hazan et~al.(2016)Hazan, Levy, and Shalev-Shwartz]{Smoothing-1}
Elad Hazan, Kfir~Yehuda Levy, and Shai Shalev-Shwartz.
\newblock On graduated optimization for stochastic non-convex problems.
\newblock In \emph{International conference on machine learning}, pages
  1833--1841, 2016.

\bibitem[Huo et~al.(2018)Huo, Gu, Jiu, and Huang]{VRSC-PG}
Zhouyuan Huo, Bin Gu, Ji~Jiu, and Heng Huang.
\newblock Accelerated method for stochastic composition optimization with
  nonsmooth regularization.
\newblock In \emph{Proceedings of the 32nd AAAI Conference on Artificial
  Intelligence}, pages 3287--3294, 2018.

\bibitem[Johnson and Zhang(2013)]{johnson2013accelerating}
Rie Johnson and Tong Zhang.
\newblock Accelerating stochastic gradient descent using predictive variance
  reduction.
\newblock In \emph{Advances in Neural Information Processing Systems 26}, pages
  315--323, 2013.

\bibitem[Karimi et~al.(2016)Karimi, Nutini, and
  Schmidt]{KarimiNutiniSchmidt2016}
Hamed Karimi, Julie Nutini, and Mark Schmidt.
\newblock Linear convergence of gradient method and proximal-gradient methods
  under the {Polyak}-{{\L}ojasiewicz} condition.
\newblock In \emph{Machine Learning and Knowledge Discovery in Database -
  European Conference, Proceedings}, pages 795--811, 2016.

\bibitem[Lei et~al.(2017)Lei, Ju, Chen, and Jordan]{LeiJuChenJordan2017}
Lihua Lei, Cheng Ju, Jianbo Chen, and Michael~I Jordan.
\newblock Non-convex finite-sum optimization via {SCSG} methods.
\newblock In \emph{Advances in Neural Information Processing Systems 30}, pages
  2348--2358. Curran Associates, Inc., 2017.

\bibitem[Li and Li(2018)]{LiLi2018NeurIPS}
Zhize Li and Jian Li.
\newblock A simple proximal stochastic gradient method for nonsmooth nonconvex
  optimization.
\newblock In \emph{Advances in Neural Information Processing Systems 31}, pages
  5564--5574. Curran Associates, Inc., 2018.

\bibitem[Lian et~al.(2017)Lian, Wang, and Liu]{SVR-SCGD}
Xiangru Lian, Mengdi Wang, and Ji~Liu.
\newblock Finite-sum composition optimization via variance reduced gradient
  descent.
\newblock In \emph{Proceedings of the 20th International Conference on
  Artificial Intelligence and Statistics (AISTATS)}, pages 1159--1167, 2017.

\bibitem[Mobahi and Fisher(2015)]{Smoothing-3}
Hossein Mobahi and John~W Fisher.
\newblock On the link between gaussian homotopy continuation and convex
  envelopes.
\newblock In \emph{International Workshop on Energy Minimization Methods in
  Computer Vision and Pattern Recognition}, pages 43--56. Springer, 2015.

\bibitem[Nesterov(2013)]{Nesterov2013composite}
Yurii Nesterov.
\newblock Gradient methods for minimizing composite functions.
\newblock \emph{Mathematical Programming}, 140\penalty0 (1):\penalty0 125--161,
  2013.

\bibitem[Nguyen et~al.(2017)Nguyen, Liu, Scheinberg, and
  Tak{\'a}{\v{c}}]{SARAH2017ICML}
Lam~M. Nguyen, Jie Liu, Katya Scheinberg, and Martin Tak{\'a}{\v{c}}.
\newblock {SARAH}: A novel method for machine learning problems using
  stochastic recursive gradient.
\newblock In Doina Precup and Yee~Whye Teh, editors, \emph{Proceedings of the
  34th International Conference on Machine Learning (ICML)}, volume~70 of
  \emph{Proceedings of Machine Learning Research (PMLR)}, pages 2613--2621,
  Sydney, Australia, 2017.

\bibitem[Nguyen et~al.(2019)Nguyen, van Dijk, Phan, Nguyen, Weng, and
  Kalagnanam]{SmoothSARAH2019}
Lam~M. Nguyen, Marten van Dijk, Dzung~T. Phan, Phuong~Ha Nguyen, Tsui-Wei Weng,
  and Jayant~R. Kalagnanam.
\newblock Finite-sum smooth optimization with sarah.
\newblock arXiv preprint, arXiv:1901.07648, 2019.

\bibitem[Pham et~al.(2019)Pham, Nguyen, Phan, and Tran-Dinh]{ProxSARAH2019}
Nhan~H. Pham, Lam~M. Nguyen, Dzung~T. Phan, and Quoc Tran-Dinh.
\newblock {ProxSARAH}: An efficient algorithmic framework for stochastic
  composite nonconvex optimization.
\newblock arXiv preprint, arXiv:1902.05679, 2019.

\bibitem[Reddi et~al.(2016{\natexlab{a}})Reddi, Hefny, Sra, Poczos, and
  Smola]{Reddi2016SVRGnonconvex}
Sashank~J. Reddi, Ahmed Hefny, Suvrit Sra, Barnabas Poczos, and Alex Smola.
\newblock Stochastic variance reduction for nonconvex optimization.
\newblock In \emph{Proceedings of The 33rd International Conference on Machine
  Learning}, volume~48 of \emph{Proceedings of Machine Learning Research},
  pages 314--323, New York, New York, USA, 2016{\natexlab{a}}.

\bibitem[Reddi et~al.(2016{\natexlab{b}})Reddi, Sra, P{\'o}czos, and
  Smola]{NCVX-SAGA-Smooth}
Sashank~J Reddi, Suvrit Sra, Barnab{\'a}s P{\'o}czos, and Alex Smola.
\newblock Fast incremental method for smooth nonconvex optimization.
\newblock In \emph{2016 IEEE 55th Conference on Decision and Control (CDC)},
  pages 1971--1977. IEEE, 2016{\natexlab{b}}.

\bibitem[Reddi et~al.(2016{\natexlab{c}})Reddi, Sra, P{\'o}czos, and
  Smola]{NCVX-SAGA-Nonsmooth}
Sashank~J Reddi, Suvrit Sra, Barnab{\'a}s P{\'o}czos, and Alexander~J Smola.
\newblock Proximal stochastic methods for nonsmooth nonconvex finite-sum
  optimization.
\newblock In \emph{Advances in Neural Information Processing Systems 29}, pages
  1145--1153, 2016{\natexlab{c}}.

\bibitem[Rockafellar(1970)]{Rockafellar70book}
R.~Tyrrell Rockafellar.
\newblock \emph{Convex Analysis}.
\newblock Princeton University Press, 1970.

\bibitem[Rockafellar(2007)]{Rockafellar2007CoherentRisk}
R.~Tyrrell Rockafellar.
\newblock Coherent approaches to risk in optimization under uncertainty.
\newblock \emph{INFORMS TutORials in Operations Research}, 2007.

\bibitem[Ruszczy\'nski(2013)]{Ruszczynski2013risk-averse}
Andrzej Ruszczy\'nski.
\newblock Advances in risk-averse optimization.
\newblock \emph{INFORMS TutORials in Operation Research}, 2013.

\bibitem[Sutton and Barto(1998)]{sutton1998reinforcement}
Richard~S Sutton and Andrew~G Barto.
\newblock \emph{Reinforcement Learning: An Introduction}.
\newblock MIT Press, Cambridge, MA, 1998.

\bibitem[Wang et~al.(2017{\natexlab{a}})Wang, Fang, and Liu]{SCGD-M.Wang}
Mengdi Wang, Ethan~X Fang, and Han Liu.
\newblock Stochastic compositional gradient descent: algorithms for minimizing
  compositions of expected-value functions.
\newblock \emph{Mathematical Programming}, 161\penalty0 (1-2):\penalty0
  419--449, 2017{\natexlab{a}}.

\bibitem[Wang et~al.(2017{\natexlab{b}})Wang, Liu, and Fang]{ASC-PG-M.Wang}
Mengdi Wang, Ji~Liu, and Ethan Fang.
\newblock Accelerating stochastic composition optimization.
\newblock \emph{Journal of Machine Learning Research}, 18\penalty0
  (105):\penalty0 1--23, 2017{\natexlab{b}}.

\bibitem[Wang et~al.(2018)Wang, Ji, Zhou, Liang, and Tarokh]{SpiderBoost2018}
Zhe Wang, Kaiyi Ji, Yi~Zhou, Yingbin Liang, and Vahid Tarokh.
\newblock {SpiderBoost}: A class of faster variance-reduced algorithms for
  nonconvex optimization.
\newblock arXiv preprint, arXiv:1810.10690, 2018.

\bibitem[Xiao and Zhang(2014)]{xiaozhang2014proxsvrg}
Lin Xiao and Tong Zhang.
\newblock A proximal stochastic gradient method with progressive variance
  reduction.
\newblock \emph{SIAM Journal on Optimization}, 24\penalty0 (4):\penalty0
  2057--2075, 2014.

\bibitem[Zhang and Xiao(2019)]{ZhangXiao2019C-SAGA}
Junyu Zhang and Lin Xiao.
\newblock A composite randomized incremental gradient method.
\newblock In \emph{Proceedings of the 36th International Conference on Machine
  Learning (ICML)}, number~97 in Proceedings of Machine Learning Research
  (PMLR), Long Beach, California, 2019.

\bibitem[Zhou et~al.(2018)Zhou, Xu, and Gu]{NestedSVRG2018NeurIPS}
Dongruo Zhou, Pan Xu, and Quanquan Gu.
\newblock Stochastic nested variance reduced gradient descent for nonconvex
  optimization.
\newblock In \emph{Advances in Neural Information Processing Systems 31}, pages
  3921--3932. Curran Associates, Inc., 2018.

\end{thebibliography}

\newpage
\appendix
\centerline{\Large \textbf{Appendices}}

\section{Convergence analysis for composite expectation case}

In this section, we focus on convergence analysis of CIVR for solving the
stochastic composite optimization problem~\eqref{prob:main}, 
and prove Theorems~\ref{theorem:cvg-general-fixed} 
and~\ref{theorem:cvg-general-adaptive}.

First, we show that under Assumption~\ref{assumption:Lip}, the composite
function $F(x)=f(g(x))$ is smooth and $F'$ has Lipschitz constant
$L_f=\ell_g^2 L_f  + \ell_f L_g$.
\begin{eqnarray*}
\left\|F'(x)-F'(y)\right\| 
&=& \left\| g'(x)^T f'(g(x)) - g'(y)^T f'(g(y))\right\| \\
&=& \left\| g'(x)^T f'(g(x)) - g'(x)^T f'(g(y)) + g'(x)^T f'(g(y)) - g'(y)^T f'(g(y))\right\| \\
&\leq& \left\| g'(x)^T f'(g(x)) - g'(x)^T f'(g(y))\right\| + \left\| g'(x)^T f'(g(y)) - g'(y)^T f'(g(y))\right\| \\
&\leq& \left\| g'(x)\right\| \left\|f'(g(x)) -  f'(g(y))\right\| + \left\|f'(g(y))\right\| \left\|g'(x) - g'(y)\right\| \\
&\leq& \left\| g'(x)\right\| \cdot L_f \left\|g(x) -  g(y)\right\| + \left\|f'(g(y))\right\| \cdot L_g \left\|x - y\right\| \\
&\leq& \ell_g L_f \ell_g \left\|x - y\right\| + \ell_f L_g \left\|x - y\right\| \\
&=& \bigl(\ell_g^2 L_f + \ell_f L_g \bigr)\|x-y\|,
\end{eqnarray*}
where we used $\|g'(x)\|\leq \ell_g$ and $\|f'(g(y))\|\leq\ell_f$,
which are implied by the Lipschitz conditions on~$g$ and $f$ respectively.

Although the incremental estimators used in CIVR are biased, 
as shown in~\eqref{eqn:biased-estimate},
we can still bound their squared distances from the targets. 
This is given in the following lemma. 
\begin{lemma}
	\label{lemma:sarah-mse}
    Suppose Assumption~\ref{assumption:Lip} holds. 
	Let $y_i^t$ and $z_i^t$ be constructed according to \eqref{defn:sarah-1} 
    and \eqref{defn:sarah-2} in Algorithm~\ref{alg:CIVR}. 
    For any $t\geq 1$ and $1\leq i\leq \tau_t-1$,
    we have the following mean squared error (MSE) bounds
   \begin{equation}
   \label{lm:sarah-mse-2}
   \begin{cases}
    \displaystyle
   ~\E\left[\|y_i^t-g(x_i^t)\|^2\right] \leq \E\left[\|y_0^t-g(x_0^t)\|^2\right] + \sum_{r=1}^{i}\frac{\ell_g^2}{S_t}\E\left[\|x_r^t-x_{r-1}^t\|^2\right] \, ,\\
    \displaystyle
   ~\E\left[\|z_i^t-g'(x_i^t)\|^2\right] \leq \E\left[\|z_0^t-g'(x_0^k)\|^2\right] + \sum_{r=1}^{i}\frac{L_g^2}{S_t}\E\left[\|x_r^t-x_{r-1}^t\|^2\right] \, .
   \end{cases}
   \end{equation}
\end{lemma}
\begin{proof}
    We first state a fact that allows us to decompose the MSE into a squared 
    bias term and a variance term, that is, for an arbitrary random vector 
    $\zeta$ and a constant vector $u$, we have
	\be
	\label{eqn:mse-decomp}
	\E[\|\zeta-u\|^2] = \|\E[\zeta]-u\|^2 + \Var(\zeta),
	\ee
	where $\Var(\zeta) := \E[\|\zeta-\E[\zeta]\|^2]$. As a result,
	\beaa
	\E\left[\|y_i^t-g(x_i^t)\|^2\big|x_i^t\right]  
    = \bigl\|\E[y_i^t|x_i^t]-g(x_i^t)\bigr\|^2 + \Var\bigl(y_i^t|x_i^t\bigr).
	\eeaa 
	For the bias term, we have $\E[y_i^t|x_i^t]-g(x_i^t) = y_{i-1}^t - g(x_{i-1}^t)$. For the variance term, we have
	\beaa
	\Var\bigl(y_i^t|x_i^t\bigr)  
    &=& \Var\biggl(y_{i-1}^t + \frac{1}{S_t}\sum_{\xi \in\cS_i^t}(g_\xi(x_i^t) - g_\xi(x_{i-1}^t)) \;\Big|\; x_i^t\biggr)\\
	&=& \frac{1}{S_t}\Var\bigl(g_\xi(x_i^t)-g_\xi(x_{i-1}^t)\;|\,x_i^t\bigr)\\ 
	&\leq& \frac{1}{S_t}\E\bigl[\|g_\xi(x_i^t) - g_\xi(x_{i-1}^t)\|^2|x_i^t\bigr]\\
	& \leq & \frac{\ell_g^2}{S_t}\|x_i^t-x_{i-1}^t\|^2,
	\end{eqnarray*}
    where the second equality is due to the fact that $y_{i-1}^t$ is a constant
    conditioning on $x_i^t$ and in the last inequality we used the 
    $\ell_g$-Lipschitz continuity of $g_\xi$. Consequently,
    \[
        \E\left[\|y_i^t-g(x_i^t)\|^2\right]  \leq  \E\left[\|y_{i-1}^t-g(x_{i-1}^t)\|^2\right] + \frac{\ell_g^2}{S_t} \E\left[\|x_i^t-x_{i-1}^t\|^2\right].
    \]
	Recursively applying the above procedure yields
	\begin{equation}
	\label{lm:sarah-mse-3}
	\E\left[\|y_i^t-g(x_i^t)\|^2\right]  \leq  \E\left[\|y_0^t-g(x_0^t)\|^2\right] + \sum_{r=1}^i\frac{\ell_g^2}{S_t} \E\left[\|x^t_r-x^t_{r-1}\|^2\right].
	\end{equation}

    Similarly, the bound on $\E\left[\|z_i^t-g'(x_i^t)\|^2\right]$ can be shown by using the $L_g$-Lipschitz continuity of~$g'_\xi$.
\end{proof}

In Algorithm~\ref{alg:CIVR}, we approximate the gradient of 
$F(x):=f(g(x))$ by $\tnf(x_i^t)=(z_i^t)^Tf'(y_i^t)$. 
The next lemma bounds the MSE of this estimator.
\begin{lemma}\label{lemma:MSE-grad-expectation}
    Suppose Assumptions~\ref{assumption:Lip} and~\ref{assumption:variance} hold.
    Then we have
\begin{equation}
\label{eqn:mse-F'}
\E[\|\tnf(x_i^t) - F'(x_i^t)\|^2] \leq \frac{G_0}{S_t}\sum_{r=1}^i\E[\|x_{r}^t - x_{r-1}^t\|^2] + \frac{\sigma_0^2}{B_t},
\end{equation}
where  
\begin{equation*}
G_0: = 2\bigl(\ell_g^4 L_f^2  + \ell_f^2 L_g^2\bigr)\quad \mbox{and}\quad 
\sigma_0^2 := 2\bigl(\ell_g^2 L_f^2\sigma_{g}^2 + \ell_f^2\sigma_{g'}^2\bigr)\,.
\end{equation*}
\end{lemma}
\begin{proof}
Using Assumption~\ref{assumption:Lip}, one immediately gets
\begin{eqnarray}
& & \E\left[\|\tnf(x_i^t) - F'(x_i^t)\|^2\right] \nonumber \\
& = & \E\left[\|(z_i^t)^Tf'(y_i^t) - (g'(x_i^t))^Tf'(g(x_i^t))\|^2\right] \nonumber \\ 
& = & \E\left[\|(z_i^t)^Tf'(y_i^t) -  (g'(x_i^t))^Tf'(y_i^t) +  (g'(x_i^t))^Tf'(y_i^t) - (g'(x_i^t))^Tf'(g(x_i^t))\|^2\right]\nonumber\\
& \leq &  2\E\left[\|(z_i^t)^Tf'(y_i^t) -  (g'(x_i^t))^Tf'(y_i^t)\|^2\right] + 2\E\left[\|(g'(x_i^t))^Tf'(y_i^t) - (g'(x_i^t))^Tf'(g(x_i^t))\|^2\right]\nonumber\\
& \leq & 2\ell_f^2\E\left[\|z_i^t-g(x_i^t)\|^2\right] + 2\ell_g^2L_f^2\E\left[\|y_i^t-g'(x_i^t)\|^2\right] .
\label{eqn:mse-F}
\end{eqnarray}
Therefore, by substituting the MSE bounds provided in Lemma \ref{lemma:sarah-mse} into inequality \eqref{eqn:mse-F}, we obtain 
\begin{eqnarray}
\E\left[\|\tnf(x_i^t) - F'(x)\|^2\right] 
&\leq& \frac{2\bigl(\ell_g^4 L_f^2+\ell_f^2 L_g^2\bigr)}{S_t}
\sum_{r=1}^i\E\left[\|x_{r}^t - x_{r-1}^t\|^2\right] \nonumber \\
&& + \; 2\ell_g^2 L_f^2 \E\left[\|y_0^t-g(x_0^t)\|^2\right] + 2\ell_f^2 \E\left[\|z_0^t-g'(x_0^k)\|^2\right] .
\label{eqn:explicit-ini-variance}
\end{eqnarray}
Under Assumption~\ref{assumption:variance}, we can bound the MSE of
the estimates in~\eqref{defn:sarah-1} as
\[
    \E\left[\|y_0^t-g(x_0^t)\|^2\right] \leq \frac{\sigma_g^2}{B_t}\,, \qquad
    \E\left[\|z_0^t-g'(x_0^k)\|^2\right] \leq \frac{\sigma_{g'}^2}{B_t} \,.
\]
Combining these MSE bounds with~\eqref{eqn:explicit-ini-variance} yields
the desired result.
\end{proof}

For the proximal gradient type of algorithms, no matter deterministic or stochastic, a common metric to quantify the optimality of~$x_i^t$ is the norm of the so-called \emph{proximal gradient mapping}
\be 
\label{defn:Gt}  
\mathcal{G}_\eta(x_i^t):=\frac{1}{\eta}\bigl(x_i^t-\hat{x}_{i+1}^t\bigr), 
\ee
where $\eta$ is the step size used to produce the update
\[
\hat{x}_{i+1}^t=\prox_r^{\eta}\bigl(x_i^t-\eta F'(x_i^t)\bigr).
\]
Since we use a constant $\eta$ throughout this paper, we will omit
the subscript~$\eta$ and use $\mathcal{G}(x)$ to denote the 
proximal gradient mapping at~$x$.

Our goal is to find a point~$x$ with 
$\E\bigl[\|\mathcal{G}(x)\|^2\bigr]\leq \epsilon$.
However, in Algorithm~\ref{alg:CIVR}, 
we only have the approximate proximal gradient mapping 
\be
\label{defn:tGt}
\tcgt:=\frac{1}{\eta}(x_i^t-x_{i+1}^t) \,,
\ee
where $x_{i+1}^t$ is computed using the estimated gradient $\tnf(x_i^t)$:
\[
x_{i+1}^t=\prox_r^{\eta}\bigl(x_i^t-\eta \tnf(x_i^t)\bigr).
\]
Hence we need to establish the connection between $\cgt$ and $\tcgt$,
which is done in the next lemma.

\begin{lemma}\label{lemma:grad-mapping-closedness}
For the two gradient mappings defined in~\eqref{defn:Gt} and~\eqref{defn:tGt},
we have
\begin{equation}  
\label{eqn:closeness}
\E\left[\|\cgt\|^2\right] \leq 2\E\left[\|\tcgt\|^2\right] + 2\E\left[\|\tnf(x_i^t)-F'(x_i^t)\|^2\right].
\end{equation}
\end{lemma}
\begin{proof}
Using the inequality
$\|x_i^t-\hat{x}_{i+1}^t\|^2\leq 2\|x_i^t-x_{i+1}^t\|^2 + 2\|x_{i+1}^t-\hat{x}_{i+1}^t\|^2$ 
and the definitions of $\cgt$ and $\tcgt$, we have
\begin{eqnarray*}
\E\left[\|\cgt\|^2\right] 
&\leq& 2\E\left[\|\tcgt\|^2\right] + \frac{2}{\eta^2}\left\|x_{i+1}^t-\hat{x}_{i+1}^t\right\|^2 \\
&=& 2\E\left[\|\tcgt\|^2\right] + \frac{2}{\eta^2}\left\|\prox_r^{\eta}\bigl(x_i^t-\eta F'(x_i^t)\bigr)-\prox_r^{\eta}\bigl(x_i^t-\eta \tnf(x_i^t)\bigr)\right\|^2 \\
&\leq& 2\E\left[\|\tcgt\|^2\right] + \frac{2}{\eta^2}\left\|x_i^t-\eta F'(x_i^t)-\bigl(x_i^t-\eta \tnf(x_i^t)\bigr)\right\|^2 \\
&=& 2\E\left[\|\tcgt\|^2\right] + 2\E\left[\|\tnf(x_i^t)-F'(x_i^t)\|^2\right],
\end{eqnarray*}
where in the second inequality we used the non-expansive property of
proximal mapping \citep[e.g.,][Section~31]{Rockafellar70book}.
\end{proof}

The next lemma bounds the amount of expected descent per iteration 
in Algorithm~\ref{alg:CIVR}.
\begin{lemma}
	\label{lemma:descent-general}
	Let the sequence $\{x_i^t\}$ be generated by Algorithm~\ref{alg:CIVR}.
	Then for all $t\geq 1$ and $0\leq i\leq \tau_t-1$, 
    we have the following two inequalities 
	\begin{equation} 
	\label{lm:descent-general-1}
	\E[\Phi(x_{i+1}^t)] ~\leq~ 
	\E[\Phi(x_i^t)] - \left(\frac{\eta}{2} - \frac{L_F\eta^2}{2}\right)
	\E\left[\|\tcgt\|^2\right] + \frac{\eta}{2}\E[\|\tnf(x_i^t) - F'(x_i^t)\|^2],
	\end{equation} 
	and
	\begin{align} 
        \E[\Phi(x_{i+1}^t)]  ~\leq~  &
	\E[\Phi(x_i^t)] - \frac{\eta}{8}\E[\|\cgt\|^2] + \frac{3\eta}{4}\E[\|\tnf(x_i^t) - F'(x_i^t)\|^2] \nonumber \\
	& - \left(\frac{1}{4\eta} - \frac{L_F}{2}\right) \E\left[\|x_{i}^t-x_{i+1}^t\|^2\right].
	\label{lm:descent-general-2}
	\end{align}  
\end{lemma} 
\begin{proof}
	By applying the $L_F$-Lipschitz continuity of $F'$ and the optimality of the $\frac{1}{\eta}$-strongly convex subproblem, we have 
	\beaa
	\Phi(x_{i+1}^t)& = & F(x_{i+1}^t) + r(x_{i+1}^t)\\
	& \leq & F(x_i^t) + \langle F'(x_i^t),x_{i+1}^t-x_i^t\rangle + \frac{L_F}{2}\|x_{i+1}^t-x_i^t\|^2 + r(x_{i+1}^t)\\
	& = & F(x_i^t) + \langle\tnf(x_i^t),x_{i+1}^t-x_i^t\rangle + \frac{1}{2\eta}\|x_{i+1}^t-x_i^t\|^2 + r(x_{i+1}^t)\\
	& & + \langle F'(x_i^t)-\tnf(x_i^t),x_{i+1}^t-x_i^t\rangle- (\frac{1}{2\eta}-\frac{L_F}{2})\|x_{i+1}^t-x_i^t\|^2 \\
	& \leq & F(x_i^t) + r(x_i^t) - \frac{1}{2\eta}\|x_{i+1}^t-x_i^t\|^2- (\frac{1}{2\eta}-\frac{L_F}{2})\|x_{i+1}^t-x_i^t\|^2 \\
	& & + \frac{\eta}{2}\|\tnf(x_i^t)-F'(x_i^t)\|^2 + \frac{1}{2\eta}\|x_{i+1}^t-x_i^t\|^2\\
	& = & \Phi(x_i^t) - (\frac{1}{2\eta}-\frac{L_F}{2})\|x_{i+1}^t-x_i^t\|^2+ \frac{\eta}{2}\|\tnf(x_i^t)-F'(x_i^t)\|^2.
	\eeaa
	Taking the expectation on both sides completes the proof of inequality \eqref{lm:descent-general-1}.  By inequality \eqref{eqn:closeness}, we know that 
	$$-\frac{\eta}{4}\E[\|\tcgt\|^2]\leq -\frac{\eta}{8}\E[\|\cgt\|^2] + \frac{\eta}{4}\E[\|\tnf(x_i^t)-F'(x_i^t)\|^2].$$
	Adding this inequality in to \eqref{lm:descent-general-1} yields \eqref{lm:descent-general-2}.
\end{proof}

\subsection{Proof of Theorem~\ref{theorem:cvg-general-fixed}}
\label{sec:proof-cvg-general-fixed}

\begin{proof}
	Because all $\tau_t$, $B_t$ and $S_t$ are taking their values independent of $t$. We denote $\tau = \tau_t$, $B = B_t$ and $S = S_t$ for all $t$ for clarity. By Lemma \ref{lemma:descent-general}, summing up inequality \eqref{lm:descent-general-2} throughout the $t$-th epoch and applying \eqref{eqn:mse-F'} gives 
	\begin{eqnarray*}
	\frac{\eta}{8}\sum_{i=0}^{\tau-1}\E[\|\cgt\|^2] & \leq & \E[\Phi(x_0^t)] - \E[\Phi(x_{\tau}^t)] - \left(\frac{1}{4\eta} - \frac{L_F}{2}\right)\sum_{r=1}^{\tau}\E[\|x_r^t-x_{r-1}^t\|^2]\\
    & &  + \frac{3G_0\eta}{4S}\sum_{i=1}^{\tau-1}\sum_{r=1}^{i}\E[\|x_r^t-x_{r-1}^t\|^2] + \frac{3\sigma_0^2\eta}{4B}\tau\\
    & \leq & \E[\Phi(x_0^t)] - \E[\Phi(x_{\tau}^t)] - \left(\frac{1}{4\eta} - \frac{L_F}{2}-\tau\frac{3G_0\eta}{4S}\right)\sum_{r=1}^{\tau}\E[\|x_r^t-x_{r-1}^t\|^2] \\
    && + \frac{3\sigma_0^2\eta}{4B}\tau,
	\end{eqnarray*}
where the second inequality is due to the fact that
\[
    \sum_{i=1}^{\tau-1}\sum_{r=1}^{i}\E[\|x_r^t-x_{r-1}^t\|^2]\leq \tau\sum_{r=1}^{\tau}\E[\|x_r^t-x_{r-1}^t\|^2].
\] 
When we choose the parameters satisfying $\tau\leq S$, then  the coefficient $\frac{1}{4\eta} - \frac{L_F}{2}-\tau\frac{3G_0\eta}{4S} \geq \frac{1}{4\eta} - \frac{L_F}{2}-\frac{3G_0\eta}{4}$ which depends only on the parameter $\eta$ and some constant. If we choose the $\eta$ according to the theorem, then  $\frac{1}{4\eta} - \frac{L_F}{2}-\frac{3G_0\eta}{4}\geq0$, yielding that 
\begin{equation}
    \label{thm:cvg-general-fixed-1}
	\frac{\eta}{8}\sum_{i=0}^{\tau-1}\E[\|\cgt\|^2] ~\leq~ \E[\Phi(x_0^t)] - \E[\Phi(x_{\tau}^t)] + \frac{3\sigma_0^2\eta}{4B}\tau.
\end{equation}
Summing this up throughout the epochs gives 
\begin{equation*}
\frac{\eta}{8}\sum_{t=1}^T\sum_{i=0}^{\tau-1}\E[\|\cgt\|^2] ~\leq~  \E[\Phi(x_0^1)] - \E[\Phi(x_{\tau}^T)] + \frac{3\sigma_0^2\eta}{4B}\tau T \leq \Phi(x_0^1) - \Phi^* + \frac{3\sigma_0^2\eta}{4B}\tau T,
\end{equation*}
where we have applied the fact that $x_0^t = x_{\tau}^{t-1}$. By the random sampling scheme for output $\bar{x}$, we have 
\begin{equation}
\label{thm:cvg-general-fixed-2}
\E[\|\mathcal{G}(\bar{x})\|^2] = \frac{1}{\tau T}\sum_{t=1}^T\sum_{i=0}^{\tau-1}\E[\|\cgt\|^2] \leq \frac{8(\Phi(x_0^1) - \Phi^*)}{\tau T\eta} + \frac{6\sigma_0^2}{B}.
\end{equation}
Substitute the values of $T,\tau$ and $B$ gives \eqref{thm:cvg-general-fixed}.
\end{proof}

To simplify presentation, we omit $\lceil\cdot\rceil$ on integer parameters
in the following discussion. 
\begin{itemize}
    \item With $\eta\leq \frac{4}{L_F+\sqrt{L_F^2+12 G_0}}$, and letting $T=1/\sqrt{\epsilon}$, $B=\sigma_0^2/\epsilon$, and $\tau=S=1/\sqrt{\epsilon}$, we have 
\[
    \E[\|\mathcal{G}(\bar{x})\|^2] \leq 8\bigl((\Phi(x_0^1)-\Phi^*)\eta^{-1} + 6 \bigr)\epsilon,
\]
and the sample complexity is 
$T(B+2\tau S)=\cO\bigl(\sigma_0^2\epsilon^{-3/2}+\epsilon^{-3/2}\bigr)$, as in our theorem.
\item With $\eta\leq\frac{4}{L_F+\sqrt{L_F^2+12 G_0}}$, and letting $T=1/\epsilon$, $B=1+\sigma_0^2/\epsilon$, and $\tau=S=1$, we again obtain
\[
    \E[\|\mathcal{G}(\bar{x})\|^2] \leq 8\bigl((\Phi(x_0^1)-\Phi^*)\eta^{-1} + 6 \bigr)\epsilon,
\]
but the sample complexity is 
$T(B+2\tau S)=\cO\bigl(\sigma_0^2\epsilon^{-2}+\epsilon^{-1}\bigr)$, 
which is same as in \citet{GhadimiLan2013}.
For deterministic optimization with $\sigma_0=0$, 
this recovers the $\cO(\epsilon^{-1})$ complexity.
\end{itemize}

\subsection{Proof of Theorem~\ref{theorem:cvg-general-adaptive}}
\label{sec:proof-cvg-general-adaptive}

\begin{proof}
	Note that for this set of parameters, we still have the relationship that  $\tau_t = S_t$. Therefore, within each epoch, \eqref{thm:cvg-general-fixed-1} is still true with epoch specific $\tau_t$ and $B_t$. Summing this up gives
	\begin{equation}
	\label{thm:cvg-general-adptive-0.5}
	\frac{\eta}{8}\sum_{t=1}^T\sum_{i=0}^{\tau_t-1}\E[\|\cgt\|^2]  \leq \Phi(x_0^1) - \Phi^* + \sum_{t=1}^T\frac{3\sigma_0^2\eta}{4B_t}\tau_t.
	\end{equation}
	By the random selection rule of $\bar{x}$, we have 
	\begin{equation}
	\label{thm:cvg-general-adptive-1}
		\E[\|\mathcal{G}(\bar{x})\|^2] = \frac{1}{\sum_{t=1}^T\tau_t}\sum_{t=1}^T\sum_{i=0}^{\tau-1}\E[\|\cgt\|^2] \leq \frac{8(\Phi(x_0^1) - \Phi^*)}{\eta\sum_{t=1}^T\tau_t} + 6\sigma_0^2\cdot\frac{\sum_{t=1}^T\tau_t/B_t}{\sum_{t=1}^T\tau_t}.
	\end{equation}
    Note that $\tau_t = \lceil at+b\rceil$ and $B_t = \lceil \sigma_0^2 (at+b)^2\rceil.$ We have
	$$\sum_{t=1}^T\tau_t \geq \sum_{t=1}^T at+b = \frac{a}{2}T(T+1) + bT = \cO(T^2)$$
	and 
	$$\sigma_0^2\sum_{t=1}^T\tau_t/B_t\leq \sum_{t=1}^T \frac{1}{at+b}\leq \frac{1}{a+b} + \int_{1}^{T}\frac{dt}{at+b} = \frac{1}{a+b}+\frac{1}{a}\ln\left(\frac{aT+b}{a+b}\right) = \cO(\ln T).$$
	Substituting the above bounds into inequality \eqref{thm:cvg-general-adptive-1} gives \eqref{thm:cvg-general-adaptive}.	As a result, 
    the total sample complexity is
\[
    \sum_{t=1}^T \bigl( B_t + 2 \tau_t S_t \bigr) \leq \sum_{t=1}^T \left(\sigma_0^2 (at+b)^2 + 2(at+b)^2\right)  
    =  \cO(\sigma_0^2 T^3 + T^3) \,.
\] 
Setting $T=\tilde{\cO}(\epsilon^{-1/2})$ so that $\E[\|\mathcal{G}\|\bar{x}\|^2]\leq\epsilon$, we get sample complexity
$\tilde{\cO}(\sigma_0^2 \epsilon^{-3/2} + \epsilon^{-3/2})$.
\end{proof}

We can also choose a different set of parameters.    
With $\eta\leq \frac{4}{L_F+\sqrt{L_F^2+12 G_0}}$, and letting $B=1+\sigma_0^2(at+b)$, and $\tau=S=1$, we also have 
\[
    \E[\|\mathcal{G}(\bar{x})\|^2] \leq \frac{8(\Phi(x_0^1)-\Phi^*)}{\eta T} + \frac{6\ln T}{T} \,,
\]
but the sample complexity, by setting $T=\tilde{\cO}(\epsilon^{-1})$ so that the above bound is less than~$\epsilon$, is
\[
    \sum_{t=1}^T \bigl( B_t + 2 \tau_t S_t \bigr) \leq \sum_{t=1}^T \left(\sigma_0^2 (at+b) + 2\right)  
    =  \cO(\sigma_0^2 T^2 + T)
    =  \tilde{\cO}(\sigma_0^2 \epsilon^{-2} + \epsilon^{-1}) \,.
\]
This is more close to the classical results on stochastic optimization.

\section{Convergence analysis for composite finite-sum case}

In this section, we consider the composite finite-sum problem~\eqref{prob:main-finite} and prove Theorems~\ref{theorem:cvg-finite-fixed} and~\ref{theorem:cvg-finite-adaptive}.

In this case, the random variable $\xi$ uniformly takes value from the finite 
index set $\{1,...,n\}$. 
At the beginning of each epoch in Algorithm~\ref{alg:CIVR}, we can choose to 
estimate $g(x_0^t)^t$ and $g'(x_0^t)$ by their exact value rather than the approximate ones constructed by subsampling. 
Namely, in \eqref{defn:sarah-1} of Algorithm \ref{alg:CIVR},
we choose $\cB_t=\{1,\ldots,n\}$ for all~$t\geq 1$.
Therefore, 
\[
y_0^t= g(x_0^t) = \frac{1}{n}\sum_{j=1}^ng_j(x_0^t),
\qquad  z_0^t= g'(x_0^t) = \frac{1}{n}\sum_{j=1}^ng'_j(x_0^t) 
\]
and
\begin{equation}\label{eqn:zero-ini-variance}
    \E\left[\|y_0^t-g(x_0^t)\|^2\right] = 0\, , \qquad
    \E\left[\|z_0^t-g'(x_0^k)\|^2\right] = 0 \,.
\end{equation}
As a result, the initial variances in Lemma~\ref{lemma:sarah-mse} diminishes
and \eqref{lm:sarah-mse-2} reduces to 
\begin{equation}
\label{eqn:sarah-mse-finite-1}
\begin{cases}
    \displaystyle
\E[\|y_i^t-g(x_i^t)\|^2] \leq \sum_{r=1}^{i}\frac{\ell_g^2}{S_t}\E[\|x_r^t-x_{r-1}^t\|^2],\\
    \displaystyle
\E[\|z_i^t-g'(x_i^t)\|^2] \leq \sum_{r=1}^{i}\frac{L_g^2}{S_t}\E[\|x_r^t-x_{r-1}^t\|^2].
\end{cases}
\end{equation}
In addition, combining~\eqref{eqn:explicit-ini-variance} 
and~\eqref{eqn:zero-ini-variance}, we have
\begin{equation}
\label{eqn:mse-F'-finite}
\E[\|\tnf(x_i^t) - F'(x)\|^2] \leq \frac{G_0}{S_t}\sum_{r=1}^i\E[\|x_{r}^t - x_{r-1}^t\|^2] .
\end{equation}
Note that Lemma \ref{lemma:descent-general} is still true. 

\subsection{Proof of Theorem~\ref{theorem:cvg-finite-fixed}}
\begin{proof}
	The proof follows similar steps as those in the proof of
    Theorem \ref{theorem:cvg-general-fixed}. 
    So we only note down the significantly different steps here. 
	
    Specifically, following the proof of Theorem \ref{theorem:cvg-general-fixed} in Section~\ref{sec:proof-cvg-general-fixed}, by applying \eqref{eqn:sarah-mse-finite-1} instead of \eqref{lm:sarah-mse-2}, we get the following result instead of inequality \eqref{thm:cvg-general-fixed-1},
    \[
	\frac{\eta}{8}\sum_{i=0}^{\tau-1}\E[\|\cgt\|^2] \leq \E[\Phi(x_0^t)] - \E[\Phi(x_{\tau}^t)].
    \]
	Summing this up apply the random selection rule of $\bar{x}$ gives
    \[
	\E[\|\mathcal{G}(\bar{x})\|^2] = \frac{1}{\tau T}\sum_{t=1}^T\sum_{i=0}^{\tau-1}\E[\|\cgt\|^2] \leq \frac{8(\Phi(x_0^1) - \Phi^*)}{\tau T\eta} = \frac{8(\Phi(x_0^1) - \Phi^*)}{\sqrt{n}T\eta}.
    \]
	Therefore, we have to set $T = \cO(\frac{1}{\sqrt{n}\epsilon})$ to get an $\epsilon$-solution. Note that the sample complexity per epoch is $n + \tau_t S_t = 2n$, the total sample complexity will be 
	$\cO(n+\sqrt{n}\epsilon^{-1})$.
\end{proof}

\subsection{Proof of Theorem~\ref{theorem:cvg-finite-adaptive}}

\begin{proof}
	If $T\leq T_0$, then the result is exactly what we proved from Theorem \ref{theorem:cvg-general-adaptive}. Therefore, the first bound in \eqref{thm:cvg-finite-adaptive} is already guaranteed.
	
	If $T>T_0$, when $1\leq t \leq T_0$, then everything still runs identically to that described in Theorem \ref{theorem:cvg-general-adaptive}. Consequently, the following bound is effective
	\begin{equation}
	\label{thm:cvg-finite-adaptive-1}
	\frac{\eta}{8}\sum_{t=1}^{T_0}\sum_{i=0}^{\tau_t-1}\E[\|\cgt\|^2]  \leq \Phi(x_0^1) - \E[\Phi(x^{T_0+1}_{0})] + \sum_{t=1}^{T_0}\frac{3\sigma_0^2\eta}{4B_t}\tau_t.
	\end{equation} 
	When $T_0+1\leq t\leq T$, the following bound becomes effective, 
	\begin{equation*}
	\frac{\eta}{8}\sum_{t=T_0+1}^T\sum_{i=0}^{\tau-1}\E[\|\cgt\|^2] \leq \E[\Phi(x_0^{T_0+1})] - \Phi^*.
	\end{equation*}
	Therefore, we have 
	\begin{eqnarray*}
	\E[\|\mathcal{G}(\bar{x})\|^2] = \frac{1}{\sum_{t=1}^T\tau_t}\sum_{t=1}^T\sum_{i=0}^{\tau-1}\E[\|\cgt\|^2] \leq \frac{8(\Phi(x_0^1) - \Phi^*)}{\eta\sum_{t=1}^T\tau_t} + 6\sigma_0^2\cdot\frac{\sum_{t=1}^{T_0}\tau_t/B_t}{\sum_{t=1}^T\tau_t}.
	\end{eqnarray*}
    Note that 
    $$\sum_{t=1}^{T_0}\tau_t/B_t\leq \sum_{t=1}^{T_0} \frac{1}{at+b}\leq\frac{1}{a+b}+\frac{1}{a}\ln\left(\frac{aT_0+b}{a+b}\right) = \cO(\ln n),$$
    and
    $$\sum_{t=1}^T\tau_t \geq (T-T_0)\sqrt{n} + \sum_{t=1}^{T_0} (at+b)  = \sqrt{n}(T - T_0) + \frac{a}{2}T^2_0+(\frac{a}{2}+b)T_0 = \cO(\sqrt{n}(T-T_0+1)).$$
    With the above two bounds, we have proved the second result in \eqref{thm:cvg-finite-adaptive}. 
    
    For any $\epsilon>0$, if $\epsilon \geq \cO(1/T_0^2) = \cO(n^{-1}).$ In this case, the algorithm will spend most epochs in the adaptive phase, whose sample complexity is $\tilde{\cO}(\epsilon^{-3/2})$. if $\epsilon=o(n^{-1})$, we need $T>T_0$. By \eqref{thm:cvg-finite-adaptive}, we know
    $\sqrt{n}(T-T_0+1) = \tilde\cO(\epsilon^{-1})$, this means that the total sample complexity will be 
    $$\sum_{t=1}^T\bigl(B_t+2\tau_t S_t\bigr) \leq 3\sum_{t=1}^{T_0} (at+b+1)^2 +  3(T-T_0)n = \tilde\cO(n^{3/2} + \sqrt{n}\epsilon^{-1}) = \tilde{\cO}(\sqrt{n}\epsilon^{-1}).$$
    When $\epsilon\geq\cO(n^{-1})$, we have $\epsilon^{-3/2}\leq \sqrt{n}\epsilon^{-1}$. When $\epsilon = o(n^{-1})$, we have $\epsilon^{-3/2} > \sqrt{n}\epsilon^{-1}$. Combining the two cases together gives the sample complexity of $\tilde{\cO}(\min\{\sqrt{n}\epsilon^{-1},\epsilon^{-3/2}\})$.
\end{proof}

\section{Convergence analysis under gradient-dominant condition}

\subsection{Proof of Theorem~\ref{theorem:cvg-KL-general}}

\begin{proof}
	Note that in this case $\Phi(x) = F(x)$. By \eqref{thm:cvg-general-fixed-2} and \eqref{defn:Grd-domi}, we have 
    \[
	\E[F(\bar{x}) - F^*] \leq \nu\E[\|F'(\bar{x})\|^2] \leq  \frac{8\nu(F(x_0^1) - F^*)}{\tau T\eta} + \frac{6\nu\sigma_0^2}{B} 
    \]
	By the selection of 
    $T = \bigl\lceil\frac{16\nu\sqrt{\epsilon}}{\eta}\bigr\rceil$, 
    $\tau=1/\sqrt{\epsilon}$ and $B=1+12\nu\sigma_0^2/\epsilon$,
    we have 
\begin{equation}\label{eqn:F-gap-exp-eps}
    \E[F(\bar{x}) - F^*] \leq \frac{1}{2}(F(x_0^1) - F^*) + \frac{1}{2}\epsilon,
\end{equation}
    which is \eqref{thm:cvg-KL-general-0}.

	Suppose we periodically restart the Algorithm \ref{alg:CIVR} after every~$T$epochs, and set the outputs to be $\bar{x}^k$, where $k = 1,2,...$ denotes
    the number of restarts.
    We use the output of the $k$th period $\bar{x}^k$ as the initial point 
    to start the next period, which produces $\bar{x}^{k+1}$.
    As a result, the above inequality translates to 
    \[
        \E[F(\bar{x}^{k+1}) - F^*] \leq \frac{1}{2}(\E[F(\bar{x}^k)] - F^*) + \frac{1}{2}\epsilon.
    \] 
Equivalently, 
\[
    \E[F(\bar{x}^{k}) - F^*]- \epsilon \leq \frac{1}{2}\left(\E[F(\bar{x}^{k-1}) - F^*]- \epsilon\right),
\] 
    which leads to
    \[
        \E[F(\bar{x}^{k}) - F^*]\leq \frac{1}{2^k}\left(\E[F(\bar{x}^{0}) - F^*]- \epsilon\right)+ \epsilon.
    \]
    Therefore, the expected optimality gap converges linearly to a $\epsilon$-ball around 0.
\end{proof}

Next we discuss the sample complexity with different parameter settings.
\begin{itemize}
    \item If we choose $\tau=S=1/\sqrt{\epsilon}$, $B_t=12\nu\sigma_0^2/\epsilon$, and $T = \bigl\lceil\frac{16\nu\sqrt{\epsilon}}{\eta}\bigr\rceil$, 
        then the total sample complexity is
        \[
            T(B+2\tau S)\ln\frac{1}{\epsilon} = \frac{16\nu\sqrt{\epsilon}}{\eta}\left(\frac{12\nu\sigma_0^2}{\epsilon} + \frac{1}{\sqrt{\epsilon}}\frac{1}{\sqrt{\epsilon}}\right)\ln\frac{1}{\epsilon} 
            =\cO\left((\nu^2\sigma_0^2\epsilon^{-1/2} + \nu\epsilon^{-1/2})\ln\epsilon^{-1}\right)
        \]
        However, the above derivation needs to assume $\frac{16\nu\sqrt{\epsilon}}{\eta}\geq 1$ or at least $\cO(1)$, which means $\epsilon>(\eta/\nu)^2$. 
        If this condition is not satisfied, then we have $T=1$ and the complexity is 
        \[
            \cO\bigl((\nu\sigma_0^2\epsilon^{-1} + \epsilon^{-1})\ln\epsilon^{-1}\bigr).
        \]
        Notice that the second term does not depend on $\nu$ or the conditions number.
    \item If we choose $\tau=S=1$, $B_t=1+12\nu\sigma_0^2/\epsilon$, and $T = \bigl\lceil\frac{16\nu}{\eta}\bigr\rceil$, 
        the we also have
        \[
        \E[F(\bar{x}) - F^*] \leq \frac{1}{2}(F(x_0^1) - F^*) + \frac{1}{2}\epsilon,\]
        and the total sample complexity is
        \[
            T(B+2\tau S)\ln\frac{1}{\epsilon} = \frac{16\nu}{\eta}\left(\frac{12\nu\sigma_0^2}{\epsilon} + 2\right)\ln\frac{1}{\epsilon} 
            =\cO\left(\nu^2\sigma_0^2\epsilon^{-1/2} + \nu\right)\ln\epsilon^{-1}
        \]
        Defining the condition number $\kappa=L_F\nu=\cO(\nu/\eta)$, the above complexity becomes
        \[
            T(B+2\tau S)\ln\frac{1}{\epsilon} 
            =\cO\left(\kappa^2\sigma_0^2\epsilon^{-1} + \kappa\right)\ln\epsilon^{-1}
        \]
        Thus when $\sigma=0$, we have $\cO\bigl(\kappa\ln\epsilon^{-1}\bigr)$ for deterministic optimization.
\end{itemize}

\subsection{Proof of Theorem~\ref{theorem:cvg-KL-finite}}
The proof is very similar to the previous one. 
It actually becomes simpler by noticing that in the finite-sum case, 
the terms involving $\sigma_0^2$ disappear:
\[
	\E[F(\bar{x}) - F^*] \leq \nu\E[\|F'(\bar{x})\|^2] \leq  \frac{8\nu(F(x_0^1) - F^*)}{\tau T\eta}.
\]
By choosing 
$T = \bigl\lceil\frac{16\nu}{\eta\sqrt{n}}\bigr\rceil$, 
    $\tau=S=\sqrt{n}$.
we again obtain~\eqref{eqn:F-gap-exp-eps}.
In this case, we have $B=n$ and
\[
    T(B+2\tau S)\epsilon^{-1}
    = \left\lceil\frac{16\nu}{\eta\sqrt{n}}\right\rceil \left(n+2\sqrt{n}\sqrt{n}\right)\ln\epsilon^{-1}
    =\cO\left(n+\nu\sqrt{n}\right)\ln\epsilon^{-1}.
\]

\section{Convergence analysis under optimally strong convexity}

In order to prove Theorems~\ref{theorem:cvg-O-S-CVX-general}
and~\ref{theorem:cvg-O-S-CVX-finite}, 
we first state Lemma~3 in \cite{xiaozhang2014proxsvrg} 
in our notations.

\begin{lemma}[Lemma 3 in \cite{xiaozhang2014proxsvrg}]
	\label{lemma:L.Xiao-T.Zhang}
	Let $\Phi(x) = F(x) + r(x)$, where $F'(x)$ is $L_F$-Lipschitz continuous, and $F(x)$ and $r(x)$ are convex. For any $x\in \mathrm{dom}(r)$, and any $v\in \R^{d}$, define
	$$x^+ := \mathrm{Prox}_{\eta r(\cdot)}(x-\eta v),\mbox{    } \mathcal{G}:=\frac{1}{\eta}(x-x^+), \mbox{ and } \Delta: = v-F'(x),$$
	where $\eta$ is a step size satisfying $0<\eta\leq 1/L_F$. Then for any $y\in\R^{d}$, 
	$$\Phi(y)\geq \Phi(x^+) + \mathcal{G}^T(y-x) + \frac{\eta}{2}\|\mathcal{G}\|^2+ \Delta^T(x^+-y).$$
\end{lemma}

\subsection{Proof of Theorem~\ref{theorem:cvg-O-S-CVX-general}}
\begin{proof}
If we set $x = x_i^t$, $y = x^*$, $v = \tnf(x_i^t)$, $x^+ = x_{i+1}^t$ and $\mathcal{G} = \tcgt$, we get the following useful inequality,
\[
\langle\tcgt,x^*-x_i^t\rangle \leq \Phi(x^*) - \Phi(x_{i+1}^t) - \frac{\eta}{2}\|\tcgt\|^2 - \langle F'(x_i^t) - \tnf(x_i^t),x^*-x_{i+1}^t\rangle.
\]
As a result we have the following inequality,
\begin{eqnarray}
& &     \|x_{i+1}^t-x^*\|^2 \nonumber \\
& = & \|x_i^t - x^*\|^2 + \eta^2\|\tcgt\|^2 +2\eta\langle \tcgt,x^*-x_i^t \rangle \nonumber\\
& \leq & \|x_i^t-x^*\|^2 + \eta^2\|\tcgt\|^2 - 2\eta(\Phi(x_{i+1}^t) - \Phi(x^*)) - \eta^2\|\tcgt\|^2 \nonumber \\
& & - 2\eta\langle F'(x_i^t) - \tnf(x_i^t),x^*-x_{i+1}^t\rangle\nonumber\\
& \leq & \|x_i^t-x^*\|^2 - 2\eta(\Phi(x_{i+1}^t) - \Phi(x^*)) + \frac{2\eta}{\mu}\|F'(x_i^t) - \tnf(x_i^t)\|^2 + \frac{\eta\mu}{2}\|x_{i+1}^t - x^*\|^2\nonumber\\
& \leq & \|x_i^t-x^*\|^2 - \eta(\Phi(x_{i+1}^t) - \Phi(x^*)) + \frac{2\eta}{\mu}\|F'(x_i^t) - \tnf(x_i^t)\|^2.
\label{eqn:C-SAGA-result}
\end{eqnarray}
Note that the inequality \eqref{eqn:C-SAGA-result} is originally obtained in 
\cite{ZhangXiao2019C-SAGA}.  
Adding $2\mu\cdot$\eqref{eqn:C-SAGA-result} to \eqref{lm:descent-general-1}, we get
\begin{eqnarray}
2\mu\eta\E[\Phi(x_{i+1}^t)-\Phi^*] & \leq & \E[\Phi(x_{i}^t) + 2\mu\|x_{i}^t - x^*\|^2] - \E[\Phi(x_{i+1}^t) + 2\mu\|x_{i+1}^t - x^*\|^2] \nonumber \\
& & -(\frac{1}{2\eta} - \frac{L_F}{2})\E[\|x_{i+1}^t-x_i^t\|^2] + \frac{9}{2}\eta\E[\|\tnf(x_i^t) - F'(x_i^t)\|^2]. 
\label{eqn:cvg-O-S-CVX-1}
\end{eqnarray}
By \eqref{eqn:cvg-O-S-CVX-1} and \eqref{eqn:mse-F'}, we have 
	\begin{eqnarray*}
	2\mu\eta\sum_{i=0}^{\tau_t-1} \E[\Phi(x_{i+1}^t) - \Phi^*] &\leq& \E[\Phi(x_{\tau_t}^t)  + 2\mu\|x_{\tau_t}^t-x^*\|^2] - \E[\Phi(x_{0}^t)  + 2\mu\|x_{0}^t-x^*\|^2] \nonumber\\
	& & - (\frac{1}{2\eta} - \frac{L_F}{2} - \tau_t\frac{9G_0\eta}{2S_t} )\sum_{r=1}^{\tau_t}\E[\|x_{r}^t-x_{r-1}^t\|^2] + \tau_t\frac{9\sigma_0^2\eta}{2B_t}.\nonumber
	\end{eqnarray*}
	According to the selection of $\tau_t,S_t,B_t$ and $\eta$, we know that the coefficient $(\frac{1}{2\eta} - \frac{L_F}{2} - \tau_t\frac{9G_0\eta}{2S_t} )\geq0$. Consequently, 
	\begin{eqnarray*}
	2\mu\eta\sum_{i=0}^{\tau_t-1} \E[\Phi(x_{i+1}^t) - \Phi^*] &\leq& \E[\Phi(x_{\tau_t}^t)  + 2\mu\|x_{\tau_t}^t-x^*\|^2] - \E[\Phi(x_{0}^t)  + 2\mu\|x_{0}^t-x^*\|^2] + \tau_t\frac{9\sigma_0^2\eta}{2B_t}.\nonumber
	\end{eqnarray*}
	Summing this up and apply the random selection rule of $\bar{x}$ gives 
	\begin{eqnarray*}
	\E[\Phi(\bar{x})-\Phi^*] &\leq& \frac{1}{2\mu\eta\tau T}\E[\Phi(x_0^1) - \Phi^* + 2\mu\|x_0^1-x^*\|^2] + \frac{9\sigma_0^2}{4\mu B_t}\nonumber\\
	 & \leq &\frac{5}{2\mu\eta\tau T}\E[\Phi(x_0^1) - \Phi^*] + \frac{9\sigma_0^2}{4\mu B_t} \,. \nonumber
	\end{eqnarray*}
    If we choose $T=\lceil \frac{5\sqrt{\epsilon}}{\mu\eta} \rceil$, $\tau=S=\frac{1}{\sqrt{\epsilon}}$ and $B_t=1+\frac{9\sigma_0^2}{2\mu\epsilon}$, then
$\frac{5}{2\mu\eta\tau T}\leq \half$ and we obtain
    \[
	\E[\Phi(\bar{x})-\Phi^*] 
    ~\leq~ \frac{1}{2}\E[\Phi(x_0^1) - \Phi^*] + \frac{1}{2}\epsilon \,. 
\]
This proves the inequality \eqref{thm:cvg-O-S-CVX-general-0}. The rest of the proof will mimic that of Theorem \ref{theorem:cvg-KL-general}.	
\end{proof}

Discussions on sample complexity:
\begin{itemize}
    \item If we choose $\tau=S=1/\sqrt{\epsilon}$, $B_t=1+\frac{9\sigma_0^2}{2\mu\epsilon}$, and $T=\lceil \frac{5\sqrt{\epsilon}}{\mu\eta} \rceil$, 
        then the sample complexity is
        \[
            T(B+2\tau S)\ln\frac{1}{\epsilon} = \frac{5\sqrt{\epsilon}}{\mu\eta}\left(\frac{9\sigma_0^2}{2\mu\epsilon} + \frac{1}{\sqrt{\epsilon}}\frac{1}{\sqrt{\epsilon}}\right)\ln\frac{1}{\epsilon} 
            =\cO\left((\mu^{-2}\sigma_0^2\epsilon^{-1/2} + \mu^{-1}\epsilon^{-1/2})\ln\epsilon^{-1}\right) \,.
        \]
        The above derivation needs to assume $\frac{5\sqrt{\epsilon}}{\mu\eta}\geq 1$ or at least $\cO(1)$, which means $\epsilon>(\eta\mu)^2$. 
        If this condition is not satisfied, then we have $T=1$ and the complexity is 
        \[
            \cO\bigl((\mu^{-1}\sigma_0^2\epsilon^{-1} + \epsilon^{-1})\ln\epsilon^{-1}\bigr).
        \]
    \item If we choose $\tau=S=1$, $B_t=1+\frac{9\sigma_0^2}{\mu\epsilon}$, and $T = \bigl\lceil\frac{5}{\mu\eta}\bigr\rceil$, 
        the we also have
        \[
        \E[F(\bar{x}) - F^*] \leq \frac{1}{2}(F(x_0^1) - F^*) + \frac{1}{2}\epsilon,\]
        and the total sample complexity is
        \[
            T(B+2\tau S)\ln\frac{1}{\epsilon} = \frac{5}{\mu\eta}\left(\frac{9\sigma_0^2}{\mu\epsilon} + 2\right)\ln\frac{1}{\epsilon} 
            =\cO\left(\mu^{-2}\sigma_0^2\epsilon^{-1} + \mu^{-1}\right)\ln\epsilon^{-1}
        \]
        Defining the condition number $\kappa=L_F\nu=\cO(1/(\mu\eta))$, the above complexity becomes
        \[
            T(B+2\tau S)\ln\frac{1}{\epsilon} 
            =\cO\left(\kappa^2\sigma_0^2\epsilon^{-1} + \kappa\right)\ln\epsilon^{-1}
        \]
        Thus when $\sigma=0$, we have $\cO\bigl(\kappa\ln\epsilon^{-1}\bigr)$ for deterministic optimization.
\end{itemize}

\subsection{Proof of Theorem~\ref{theorem:cvg-O-S-CVX-finite}}
The proof is very similar to the previous one. 
It actually becomes simpler by noticing that in the finite-sum case, 
the terms involving $\sigma_0^2$ disappear.

\clearpage

\section{Numerical experiments on policy evaluation for MDP}

\begin{figure*}[t]
	\centering 
	\includegraphics[width=0.345\linewidth]{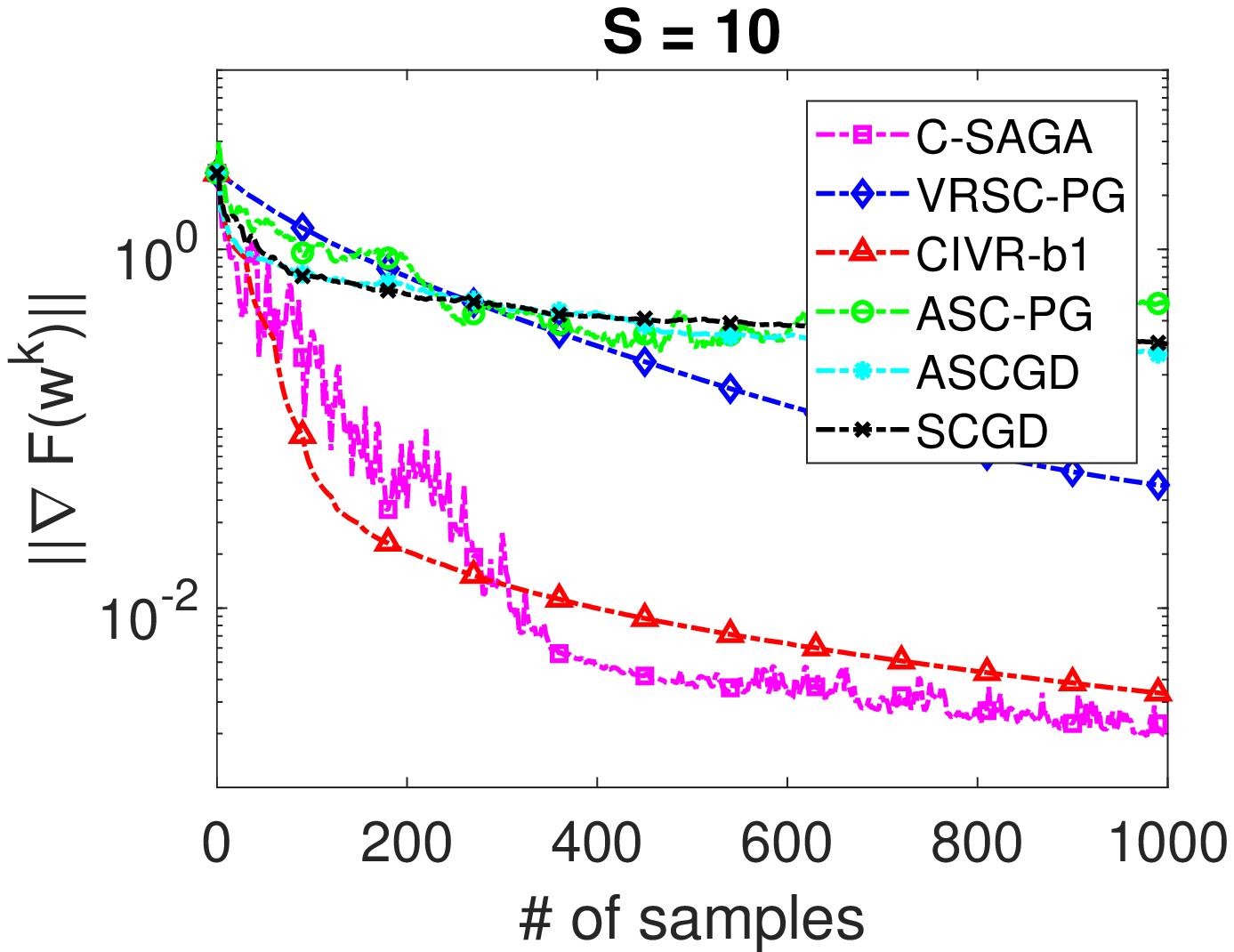}
	\hspace{-0.03\linewidth}  
	\includegraphics[width=0.345\linewidth]{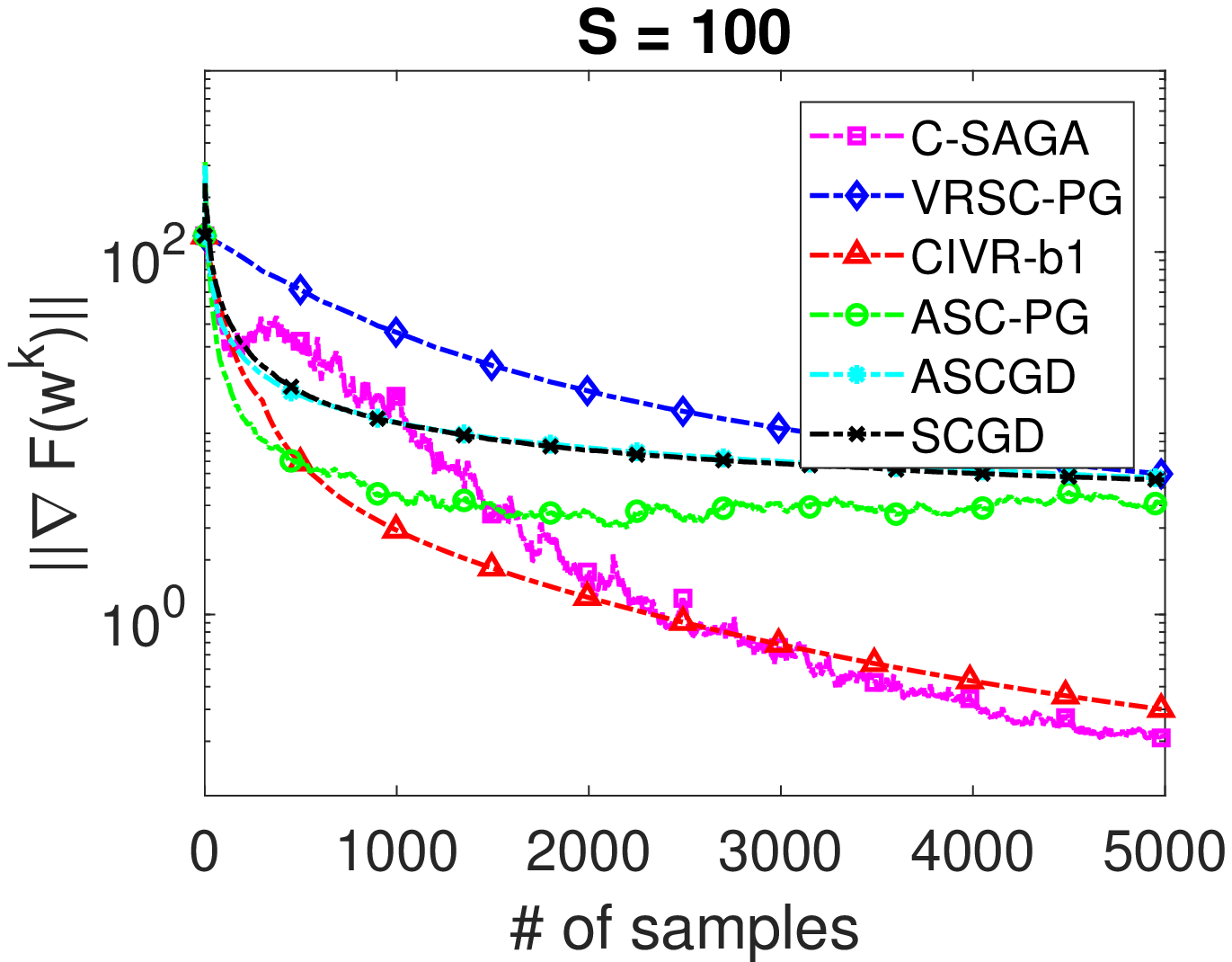}
	\hspace{-0.03\linewidth}  
	\includegraphics[width=0.345\linewidth]{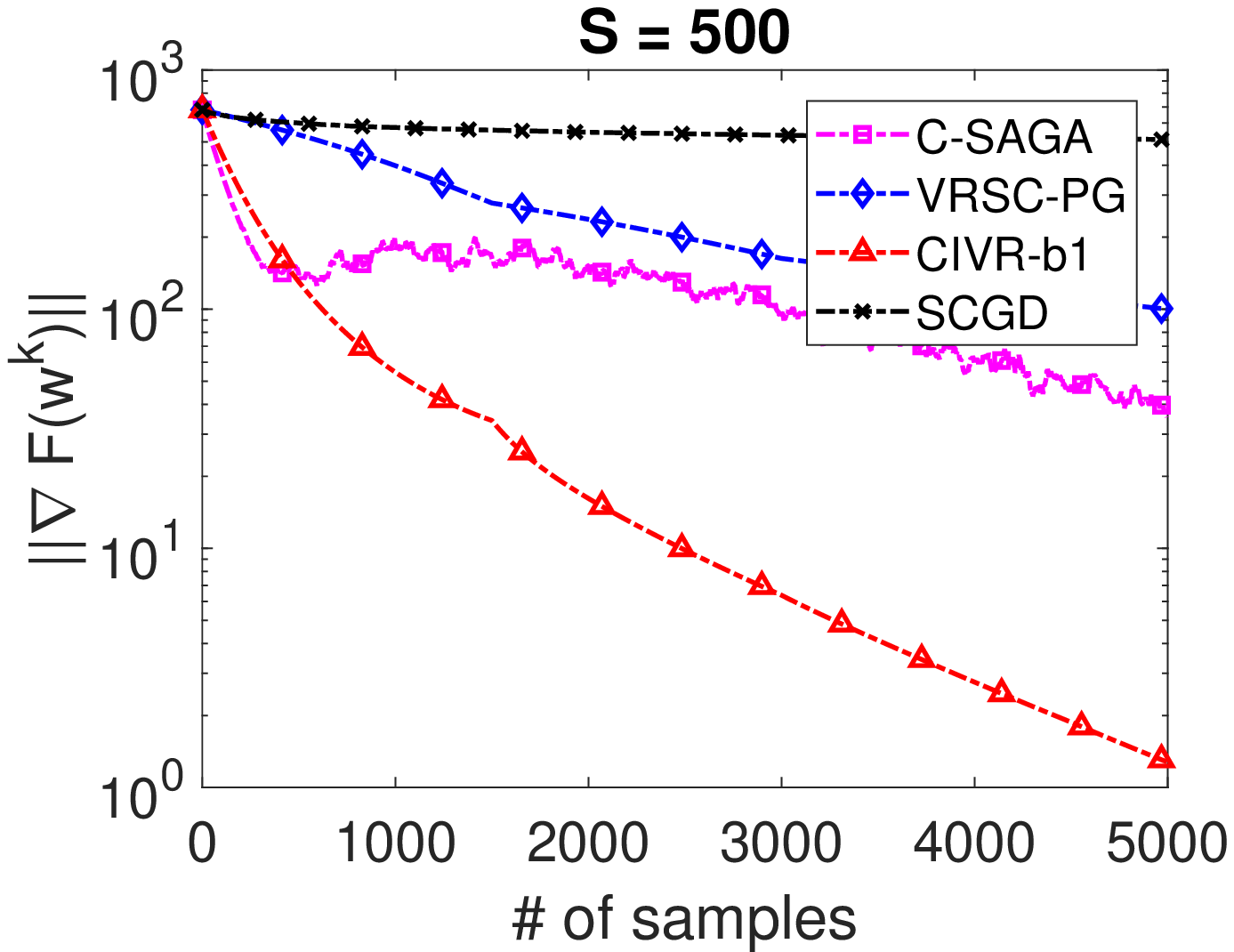}
	\vfill  
	\includegraphics[width=0.345\linewidth]{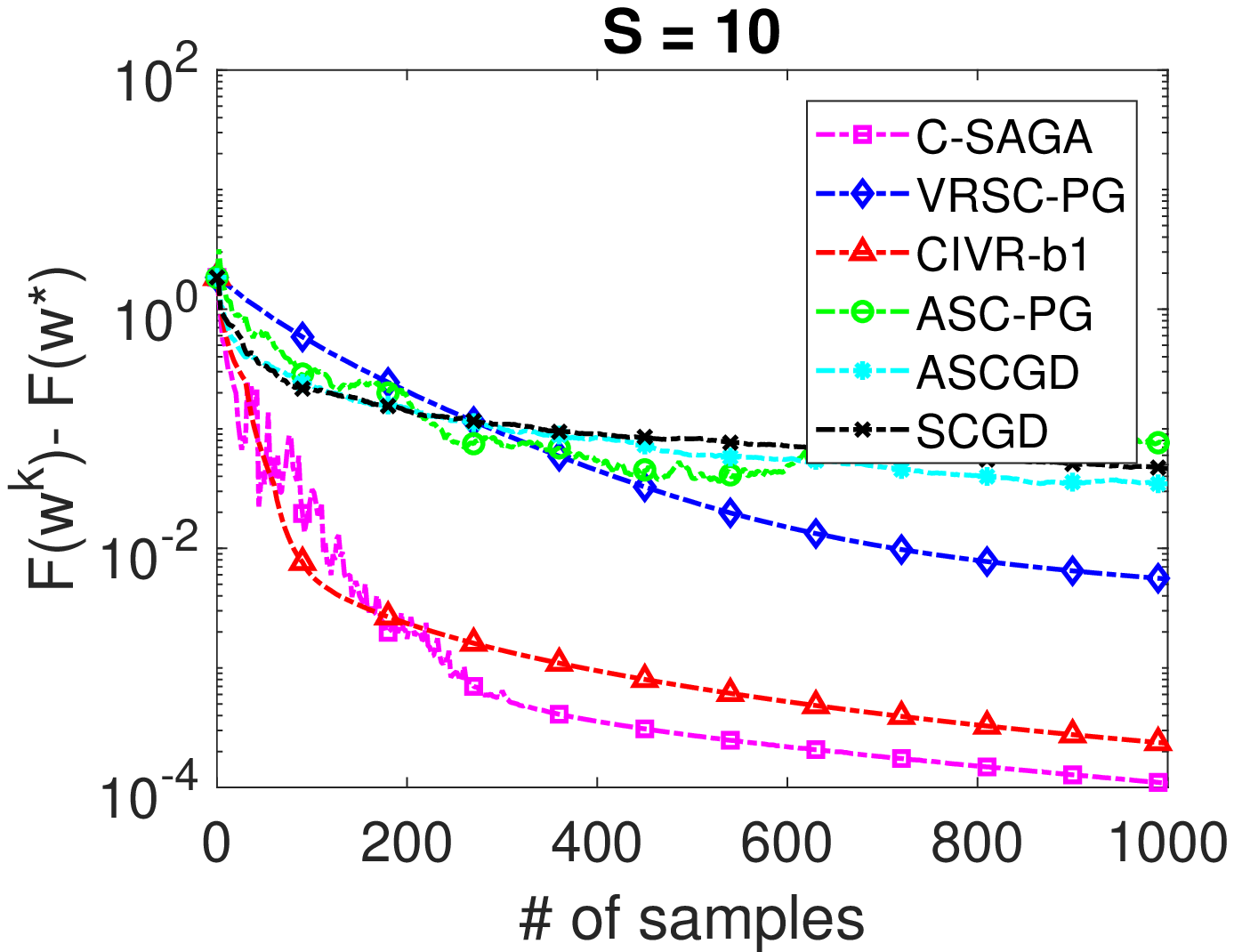}
	\hspace{-0.03\linewidth}  
	\includegraphics[width=0.345\linewidth]{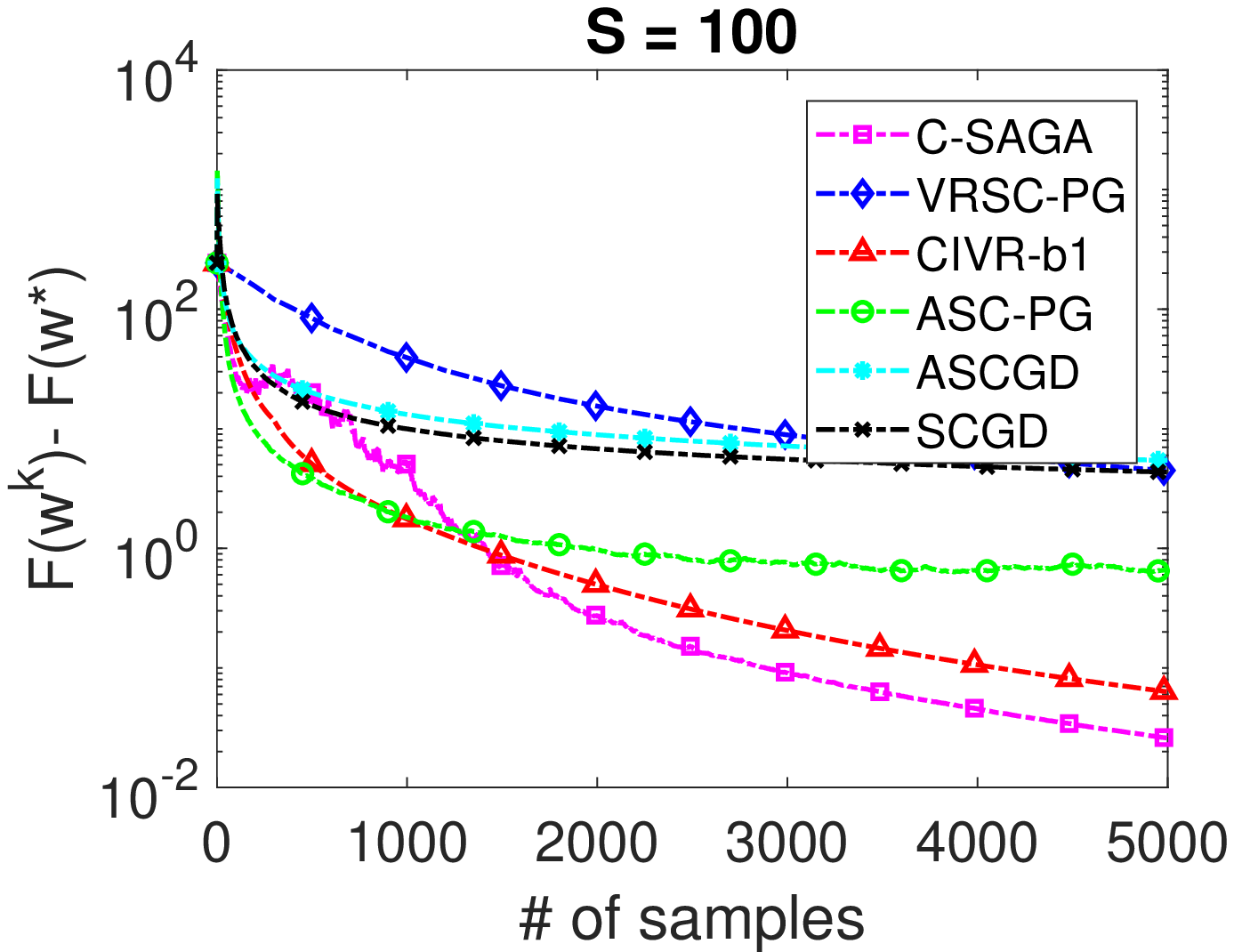}
	\hspace{-0.03\linewidth}  
	\includegraphics[width=0.345\linewidth]{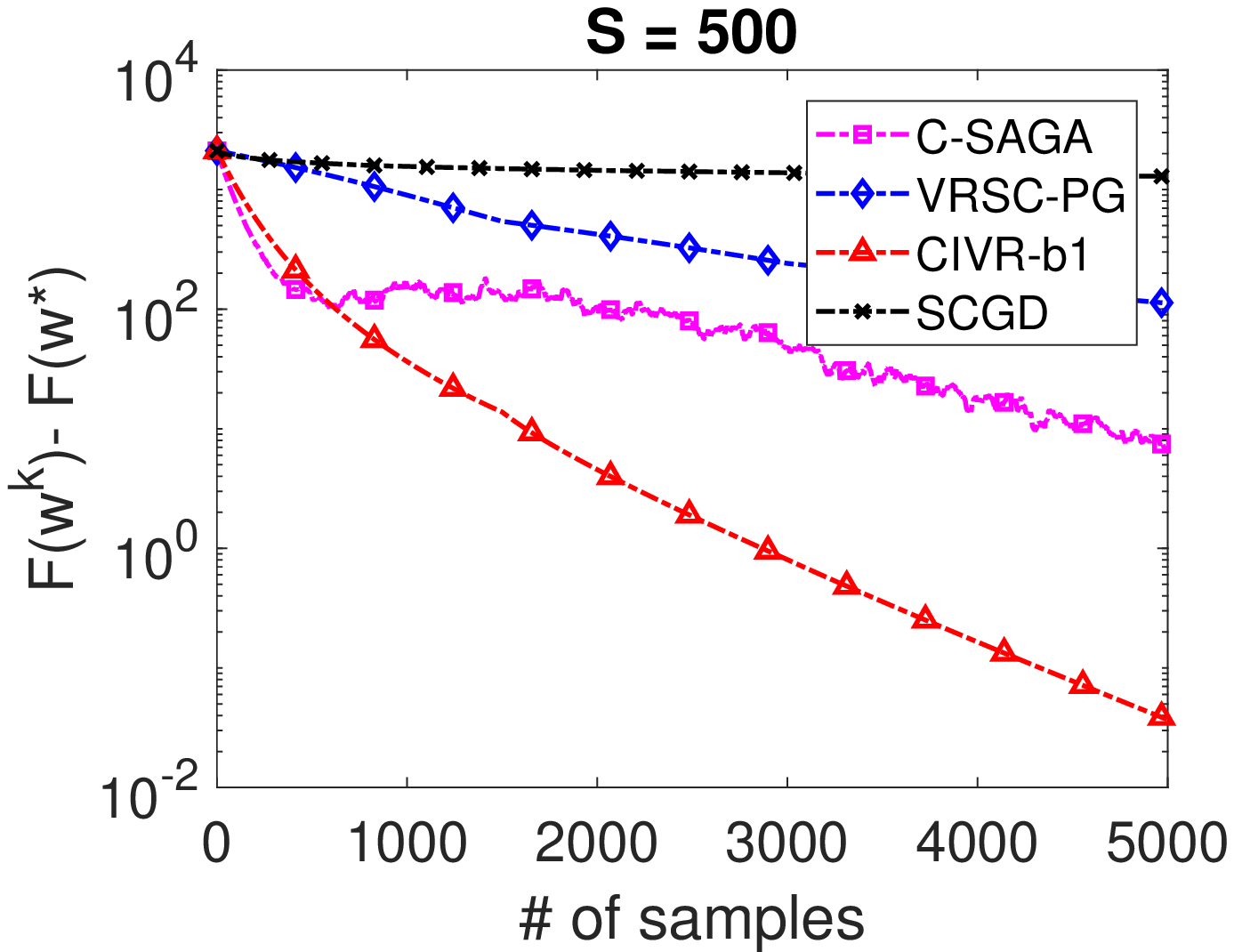}
	\caption{Experiments on policy evaluation  for MDP for 
		cases with $S=10$,  $S=100$ and  $S=500$.}
	\label{fig:MDP}
\end{figure*}

Here we provide additional numerical experiments on the policy
evaluation problem for MDP.

Let $\cS = \{1,...,S\}$ be the state space of some Markov decision process. Suppose a reward of $R_{i,j}$ is received after transitioning
from state~$i$ to state~$j$. 
Let $P^{\pi}\in\R^{S\times S}$ be the transition probability matrix under some fixed policy~$\pi$. Then the evaluation of the value function $V^\pi:\cS\to\R$ under such policy is equivalent to solving the following Bellman equation:
$$V^\pi(i) =  \sum_{j=1}^SP^\pi_{i,j}(R_{i,j} + \gamma V^\pi(j)) = \E_{j|i}[R_{i,j} + \gamma V^\pi(j)].$$
Following the suggestion of \cite{dann2014policy, ASC-PG-M.Wang}, we apply the linear function approximation 
$V^\pi(i) \approx \langle\Psi_i,w^*\rangle$ for a given set of feature vectors $\Psi_i$. 
and would like to compute the optimal vector $w^*$. 
This can be formulated as the following problem 
\bee
\label{prob:Bellman}
\minimize_{w} ~F(w) \triangleq \sum_{i=1}^S  \bigg(\langle\Psi_i,w\rangle - \sum_{j=1}^SP^\pi_{i,j}(R_{i,j} + \gamma \langle\Psi_j,w\rangle)\bigg)^2.
\eee
Let's denote 
$$q^\pi_i(w) \triangleq \sum_{j=1}^SP^\pi_{i,j}(R_{i,j} + \gamma \langle\Psi_j,w\rangle) = \E_{j|i}[R_{i,j} + \gamma \langle\Psi_j,w\rangle].$$
Then by defining 
\[
g(w) = \left[\langle\Psi_1,w\rangle,...,\langle\Psi_S,w\rangle,q^\pi_i(w),...,q^\pi_S(w)\right]^T
\]
and 
\[
f(y_1,...,y_S, z_1,...,z_S) 
= \|y-z\|^2
= \sum_{i=1}^S(y_i-z_i)^2,
\]
the Least squares problem is transformed into the form of~\eqref{prob:main-finite}. 

For this problem, we test the  SCGD \cite{SCGD-M.Wang}, the ASCGD
\cite{SCGD-M.Wang}, the ASC-PG \cite{ASC-PG-M.Wang}, the VRSC-PG
\cite{VRSC-PG}, C-SAGA \cite{ZhangXiao2019C-SAGA} and our CIVR algorithms.
In Section~\ref{sec:experiments}, we already tested the algorithms under their
standard batch sizes, e.g. $\lceil n^{2/3}\rceil$ and $\lceil\sqrt{n}\rceil$.
However, small constant batch sizes are often preferred in practice. Therefore,
we would like to set the batch size to $s=1$ for all algorithms. For this
special case, we denote the CIVR as the CIVR-b1. To balance the sample
complexity between the initial full batch sampling and the later subsampling 
with $s=1$, we set the epoch length for VRSC-PG and CIVR-b1 to be $S$. 

Note that the last $S$ components of $g$ are all independent expectations, 
therefore the variance reduction technique of VRSC-PG \cite{VRSC-PG}, 
C-SAGA \cite{ZhangXiao2019C-SAGA} and CIVR-b1 applied to each of these 
components.
In the experiments, $P^\pi$, $\Phi$ and $R^\pi$ are generated randomly. 

Similar to the experiments performed in Section~\ref{sec:experiments}, 
the step sizes are chosen from $\{0.1, 0.05,
0.01, 0.005, 0.001, 0.0005, 0.0001\}$ by experiments for VRSC-PG, C-SAGA as
well as for CIVR-b1. For $S=10$, $\eta = 0.1$ works best for both C-SAGA and
CIVR-b1, while $\eta = 0.01$ works best for VRSC-PG;  For $S=100$, $\eta =
0.001$ works best for both C-SAGA and  CIVR-b1, while $\eta = 0.0001$ works
best for VRSC-PG. For $S=500$, $\eta = 0.0001$ works best for all three of
them. 

When $S = 10$ and $S=100$,  we choose $\alpha_k = 0.01k^{-3/4}$ and $\beta_k =
0.1k^{-1/2}$ for SCGD,  $\alpha_k = 0.01k^{-5/7}$ and $\beta_k = 0.1k^{-4/7}$
for ASCGD and   $\alpha_k = 0.01k^{-1/2}$ and $\beta_k = 0.1k^{-1}$ for ASC-PG.
When $S = 500$, we choose $\alpha_k = 0.0001k^{-3/4}$ and $\beta_k =
0.001k^{-1/2}$ for SCGD while ASCGD and ASC-PG fail to converge under various
trials of parameters. The meaning of these step size parameters can be found in
\cite{ASC-PG-M.Wang} and \cite{SCGD-M.Wang}. 

Figure~\ref{fig:MDP} shows three experiments with sizes $S = 10$, $S = 100$ and $S = 500$ respectively. 
We can see that both C-SAGA and CIVR-b1 preform much better than other 
algorithms in our setting.
CIVR-b1 has more smooth and stable trajectory than C-SAGA.


\end{document}